\documentclass[11pt]{article}

\title{The embedded contact homology index revisited}
 \author{Michael Hutchings\footnote{Partially supported by NSF grant
     DMS-0505884}}
\date{}

\usepackage{amssymb}
\usepackage{latexsym}
\usepackage{amsmath}
\usepackage{amsthm}
\usepackage{amscd}
\usepackage[dvips]{graphics}

\numberwithin{equation}{section}

\newcommand{\mc}[1]{{\mathcal #1}}

\newtheorem{theorem}{Theorem}[section]
\newtheorem{proposition}[theorem]{Proposition}
\newtheorem{corollary}[theorem]{Corollary}
\newtheorem{lemma}[theorem]{Lemma}
\newtheorem{lemma-definition}[theorem]{Lemma-Definition}

\theoremstyle{definition}
\newtheorem{definition}[theorem]{Definition}
\newtheorem{remark}[theorem]{Remark}

\newtheorem{example}[theorem]{Example}

\newtheorem{notation}[theorem]{Notation}

\newtheorem*{acknowledgments}{Acknowledgments}

\newcommand{\floor}[1]{\left\lfloor #1 \right\rfloor}
\newcommand{\ceil}[1]{\left\lceil #1 \right\rceil}

\newcommand{\pin}{P^{\op{in}}}
\newcommand{\pout}{P^{\op{out}}}
\newcommand{\lin}{\Lambda^{\op{in}}}
\newcommand{\lout}{\Lambda^{\op{out}}}

\newcommand{\eqdef}{\;{:=}\;}

\renewcommand{\frak}{\mathfrak}

\newcommand{\C}{{\mathbb C}}

\newcommand{\R}{{\mathbb R}}

\newcommand{\Z}{{\mathbb Z}}
\newcommand{\op}{\operatorname}

\newcommand{\Spinc}{\op{Spin}^c}

\newcommand{\Ker}{\op{Ker}}

\newcommand{\bpm}{\begin{pmatrix}}
\newcommand{\epm}{\end{pmatrix}}

\begin{document}

\setcounter{tocdepth}{2}

\maketitle

\begin{abstract}
    Let $Y$ be a closed oriented 3-manifold with a contact form such
  that all Reeb orbits are nondegenerate.  The embedded contact
  homology (ECH) index associates an integer to each relative
  $2$-dimensional homology class of surfaces whose boundary is the
  difference between two unions of Reeb orbits.  This integer
  determines the relative grading on ECH; the ECH differential counts
  holomorphic curves in the symplectization of $Y$ whose relative
  homology classes have ECH index $1$.  A known index inequality
  implies that such curves are (mostly) embedded and satisfy some
  additional constraints.

  In this paper we prove four new results about the ECH index.  First,
  we refine the relative grading on ECH to an absolute grading, which
  associates to each union of Reeb orbits a homotopy class of oriented
  2-plane fields on $Y$.  Second, we extend the ECH index inequality
  to symplectic cobordisms between three-manifolds with stable
  Hamiltonian structures, and simplify the proof.  Third, we establish
  general inequalities on the ECH index of unions and multiple covers
  of holomorphic curves in cobordisms.  Finally, we define a new
  relative filtration on ECH, or any other kind of contact homology of
  a contact 3-manifold, which is similar to the ECH index and related
  to the Euler characteristic of holomorphic curves.  This does not
  give new topological invariants except possibly in special
  situations, but it is a useful computational tool.
\end{abstract}

\tableofcontents

\section{Introduction}
\label{sec:intro}

We begin with a very brief overview of embedded contact homology, and
then describe the results of this paper.  More detailed definitions will
be given later.

\subsection{Embedded contact homology}

Let $Y$ be a closed oriented 3-manifold with a contact form $\lambda$
such that all Reeb orbits are nondegenerate. Let
$\xi\eqdef\Ker(\lambda)$ denote the associated contact structure, and
let $\Gamma\in H_1(Y)$.  The embedded contact homology
$ECH_*(Y,\xi,\Gamma)$ is the homology of a chain complex which is
generated by finite sets of pairs $\alpha=\{(\alpha_i,m_i)\}$, where
the $\alpha_i$'s are distinct embedded Reeb orbits, the $m_i$'s are
positive integers, $\sum_im_i[\alpha_i]=\Gamma \in H_1(Y)$, and
$m_i=1$ whenever $\alpha_i$ is hyperbolic.  The differential
$\partial$ on the chain complex counts certain (mostly) embedded
holomorphic curves in $\R\times Y$, with respect to a suitable
$\R$-invariant almost complex structure $J$.

More precisely, the differential counts holomorphic curves $C$ whose
ECH index equals one.  The ECH index $I(C)$, originally defined in
\cite{pfh2} and reviewed here in \S\ref{sec:I}, is a certain
topological quantity which depends only on the relative homology class
of $C$.  The relation between the condition $I(C)=1$ and embeddedness
is as follows.  It is shown\footnote{The index inequality
  \eqref{eqn:II} was proved in a different and easier context in
  \cite{pfh2}.  To carry over the argument to the present setting, one
  needs to use the asymptotic analysis of Siefring \cite{siefring1},
  see \S\ref{sec:IICob}.}  in \cite{pfh2} that if $C$ is not multiply
covered, then
\begin{equation}
\label{eqn:II}
\op{ind}(C) \le I(C) - 2\delta(C),
\end{equation}
where $\delta(C)$ is a nonnegative integer which equals zero if and
only if $C$ is embedded.  Here $\op{ind}(C)$ denotes the Fredholm
index of $C$, which is the dimension of the moduli space of
holomorphic curves near $C$ if $J$ is generic.  It is further shown in
\cite{pfh2} that if $T$ is a union of $\R$-invariant cylinders, and if
the image of $C$ contains no $\R$-invariant cylinder, then
\begin{equation}
\label{eqn:trivCyl}
I(C\cup T) \ge I(C) + 2\#(C\cap T),
\end{equation}
where `$\#$' denotes the algebraic intersection number, which is
nonnegative by intersection positivity.  The inequalities
\eqref{eqn:II} and \eqref{eqn:trivCyl} imply, as explained in
\cite[Cor. 11.5]{t3}, that if $J$ is generic, then any holomorphic
curve $C$ with $I(C)=1$ consists of an embedded component of Fredholm
index one, possibly together with some covers of $\R$-invariant
cylinders which do not intersect the rest of $C$.  These are the
curves that the ECH differential counts. In particular, $I$ defines
the relative grading on the chain complex.  See \cite{t3} for more
about ECH, and \cite[\S7]{obg1} for a proof that $\partial^2=0$.

A priori, ECH might depend not only on $Y$, $\xi$, and $\Gamma$, but
also on the choice of contact form $\lambda$ and almost complex
structure $J$.  However, Taubes \cite{echswf} has recently shown that,
as conjectured in \cite{t3}, ECH is not only independent of $\lambda$
and $J$, but also isomorphic to a version of Seiberg-Witten Floer
homology as defined by Kronheimer-Mrowka \cite{km}.  The precise
statement is that\footnote{Taubes replaces the r.h.s.\ of
  \eqref{eqn:echswf} with the isomorphic group
  $\widehat{HM}^{-*}(Y,\frak{s}(\xi) + \op{PD}(\Gamma))$.  This is
  also isomorphic to the completed version
  $\check{HM}_\bullet(-Y,\frak{s}(\xi)+\op{PD}(\Gamma))$, and
  conjecturally isomorphic to the Heegaard Floer homology
  $HF^+_*(-Y,\frak{s}(\xi)+\op{PD}(\Gamma))$.}
\begin{equation}
\label{eqn:echswf}
ECH_*(Y,\xi,\Gamma) \simeq \check{HM}_*(-Y,\frak{s}(\xi)+\op{PD}(\Gamma)),
\end{equation}
up to a grading shift, where $\frak{s}(\xi)$ is a spin-c structure
determined by $\xi$, see \S\ref{sec:TP}.  Thus ECH is more or less a
topological invariant of the three-manifold $Y$, and in this regard it
differs substantially from the symplectic field theory of
Eliashberg-Givental-Hofer \cite{e,egh}, which is highly sensitive to the
contact structure and vanishes for overtwisted ones \cite{bn,yau}. The
isomorphism \eqref{eqn:echswf} can be regarded as an extension of
Taubes's ``Seiberg-Witten=Gromov'' theorem for closed symplectic
4-manifolds \cite{swgr} to the noncompact symplectic 4-manifold
$\R\times Y$.  This hoped-for correspondence was the original
motivation for the definition of ECH, see \cite{pfh2}.

\subsection{New results on the ECH index}

Despite this motivation, the definition of ECH, and especially the ECH
index, may at first seem a bit strange.  The aim of this paper is to shed
some additional light on the ECH index by proving four new results about it.

\subsubsection{Absolute grading}

First, in \S\ref{sec:AG} we show that the relative grading on ECH can
be refined to an absolute grading, which associates to each generator
a homotopy class of oriented 2-plane fields on $Y$, see
Theorem~\ref{thm:AG}.  If $\alpha=\{(\alpha_i,m_i)\}$ is an ECH
generator, then the associated 2-plane field $I(\alpha)$ is obtained
by modifying the contact plane field $\xi$ in a canonical manner (up
to homotopy, depending only on $m_i$) in disjoint tubular
neighborhoods of the Reeb orbits $\alpha_i$.

Recall from \cite{km} that Seiberg-Witten Floer homology also has an
absolute grading by homotopy classes of oriented 2-plane fields.  We
conjecture that Taubes's isomorphism \eqref{eqn:echswf} between ECH
and Seiberg-Witten Floer homology respects these absolute
gradings.

We also expect that one can define a similar absolute grading on
Heegaard Floer homology, by refining the construction in
\cite[\S2.6]{os} that associates to each Heegaard-Floer generator a
spin-c structure.

\subsubsection{Index inequality in cobordisms}

Second, in \S\ref{sec:IICob} we generalize the index inequality
\eqref{eqn:II} to holomorphic curves in four-dimensional symplectic
cobordisms, see Theorem~\ref{thm:indI}. Our proof follows the original
proof of \eqref{eqn:II}, but with a new and simpler proof of the key
combinatorial lemma.

\subsubsection{Unions and multiple covers}

Third, in \S\ref{sec:IU} we prove a new inequality on the ECH index of
unions and multiple covers of holomorphic curves in cobordisms, see
Theorem~\ref{thm:IU}.  This inequality is a substantial generalization
of \eqref{eqn:trivCyl} and asserts that if $C$ and $C'$ are two
holomorphic curves, then
\begin{equation}
\label{eqn:ICC'}
I(C\cup C') \ge I(C) + I(C') + 2 C\cdot C',
\end{equation}
where $C\cdot C'$ is an ``intersection number'' of $C$ and $C'$
defined in \S\ref{sec:IN}.  If the images of $C$ and $C'$ do not have
any irreducible components in common, then $C\cdot C'$ is simply the
algebraic count of intersections of $C$ and $C'$, which is nonnegative
by intersection positivity.  If the images of $C$ and $C'$ have a
common component, then the definition of $C\cdot C'$ is more subtle.
In particular, $C \cdot C$ can be negative.

Ultimately, when $X$ is a symplectic cobordism from $Y_+$ to $Y_-$,
one would like to define a map from the ECH of $Y_+$ to that of $Y_-$
by counting holomorphic curves $C$ in $X$ with ECH index $I(C)=0$.  A
major difficulty is that even if $J$ is generic, an arbitrary $I=0$
curve may contain some negative ECH index multiple covers, together
with some other components with positive Fredholm index.  The
inequality \eqref{eqn:ICC'} clarifies the extent to which this can
happen.  Note that this problem does not arise in defining the ECH
differential.  Indeed, if $X$ is a symplectization $\R\times Y$ with
an $\R$-invariant almost complex structure, then with some trivial
exceptions $C\cdot C$ is always nonnegative, see
Proposition~\ref{prop:CdotC}.

\subsubsection{Euler characteristic and relative filtration}

While the ECH differential counts holomorphic curves $C$ with
$I(C)=1$, the latter condition does not specify the genus or Euler
characteristic of $C$.  To complete the picture here, the last part of
this paper introduces another relative index, which we denote by
$J_0$.  This is a natural cousin of the ECH index $I$, and has similar
basic properties.  An analogue of the inequality \eqref{eqn:II} holds
for $J_0$, in which $J_0$ bounds the negative Euler characteristic
instead of the Fredholm index, see Corollary~\ref{cor:EulerBound} and
the stronger Proposition~\ref{prop:JBound}.  A version of the
inequality \eqref{eqn:ICC'} also holds for $J_0$, see
Proposition~\ref{prop:JU}.  The resulting bound on the topological
complexity of holomorphic curves in terms of $J_0$ plays a key role in
a subsequent paper \cite{wh}, which obtains various extensions of
the Weinstein conjecture.

The above inequalities also lead to the last main
result of the present paper, Theorem~\ref{thm:J+}, asserting that if
$X$ is the symplectization of a contact manifold $Y$ with an
$\R$-invariant almost complex structure, then every holomorphic curve
$C$ in $X$ satisfies $J_+(C)\ge 0$.  Here $J_+$ is another relative
index which is a slight variant of $J_0$.  It follows that $J_+$
defines a relative filtration on embedded contact homology, or for
that matter on any kind of contact homology of a contact 3-manifold.
As explained in \S\ref{sec:JDiscussion}, this filtration is a useful
computational tool, although it does not give new topological
invariants except possibly in special situations.

\subsubsection{Stable Hamiltonian structures}

Embedded contact homology is very similar to the periodic Floer
homology (PFH) of mapping tori considered in \cite{pfh2,pfh3}.  In
fact, there is a more general geometric structure from \cite{behwz},
called a ``stable Hamiltonian structure'', which includes both contact
manifolds and mapping tori as special cases, and for which one still
has Gromov-type compactness for holomorphic curves.  The definition of
ECH or PFH then extends in a straightforward way to any 3-manifold
with a stable Hamiltonian structure in which all Reeb orbits are
nondegenerate\footnote{When the stable Hamiltonian structure is not
  contact, one needs to either assume a ``monotonicity'' condition as
  in \cite[\S2]{pfh3}, or work with coefficients in a suitable Novikov
  ring.}.  For this reason, we will use stable Hamiltonian structures
as the basic geometric setup throughout this paper.

%

\section{The ECH index}
\label{sec:I}

We now review the definition of the ECH index, and the various notions
that enter into it, in the context of stable Hamiltonian structures.

\subsection{Stable Hamiltonian structures}

Let $Y$ be an oriented $3$-manifold.  For simplicity we assume that
$Y$ is closed, although for most of this paper this is not actually
necessary.

\begin{definition}
  \cite{behwz,cm,siefring1} A {\bf stable Hamiltonian structure\/} on
  $Y$ is a pair $(\lambda,\omega)$, where $\lambda$ is a $1$-form on
  $Y$, and $\omega$ is a $2$-form on $Y$, such that:
\begin{gather*}
\lambda\wedge\omega  > 0,\\
d\omega = 0,\\
\Ker(\omega) \subset \Ker(d\lambda).
\end{gather*}
\end{definition}

A stable Hamiltonian structure determines an oriented $2$-plane field
\[
\xi \eqdef \Ker(\lambda).
\]
It also determines a vector field $R$ defined by
\[
\omega(R,\cdot) = 0, \quad \quad \lambda(R)=1.
\]
We will call $R$ the {\bf Reeb vector field\/}, and the flow
determined by $R$ the {\bf Reeb flow\/}.  The definition of
stable Hamiltonian structure implies that $R$ is transverse to $\xi$, the
restriction of $\omega$ to $\xi$ is nondegenerate, and the Reeb flow
preserves the stable Hamiltonian structure, i.e.\ $\mc{L}_R\lambda = 0$ and
$\mc{L}_R\omega = 0$.

\begin{example}
If $\lambda$ is a contact $1$-form on $Y$, i.e. $\lambda\wedge
d\lambda >0$, then $(\lambda,d\lambda)$ is a stable Hamiltonian structure,
in which $\xi$ is the contact $2$-plane field, and $R$ is the Reeb
vector field in the usual sense.
\end{example}

\begin{example}
  Let $\Sigma$ be a surface with a symplectic form $\omega$,
  and let $\phi:(\Sigma,\omega)\to(\Sigma,\omega)$ be a
  symplectomorphism.  Let $Y$ be the {\bf mapping torus\/}
\[
Y \eqdef \frac{[0,1]\times\Sigma}{(1,x)\sim(0,\phi(x))}.
\]
Projection onto the $[0,1]$ factor defines a fiber bundle $\pi:Y\to
S^1$.  Let $t$ denote the $[0,1]$ coordinate.  The vector field
$\partial_t$ on $[0,1]\times\Sigma$ descends to a vector field on $Y$,
which we also denote by $\partial_t$.  The $2$-forms $\omega$ on the
fibers of $Y$ extend to a closed $2$-form $\omega_Y$ on $Y$ which
annihilates $\partial_t$.  Then $(\pi^*dt,\omega_Y)$ is a stable Hamiltonian
structure on $Y$, in which $\xi$ is the vertical tangent bundle of
$\pi$, and $R=\partial_t$.
\end{example}

\subsection{Reeb orbits}

Fix a closed oriented 3-manifold $Y$ with a stable Hamiltonian structure
$(\lambda,\omega)$.  A {\bf Reeb orbit\/} is a closed orbit of the
Reeb flow, i.e.\ a smooth map $\gamma:\R/T\to Y$ for some $T>0$ such
that $\gamma'(t)=R(\gamma(t))$.  Two Reeb orbits are considered the
same if they differ only by precomposition with a rotation of $\R/T$.
Given a Reeb orbit $\gamma:\R/T\to Y$ and a positive integer $k$, the
$k$-fold {\bf iterate\/} of $\gamma$ is the pullback of $\gamma$ to
$\R/kT$, which we denote by $\gamma^k$.

Given a Reeb orbit $\gamma$, for any $y$ in the
image of $\gamma$, the linearization of the Reeb flow along $\gamma$
defines a symplectic linear map
\[
P_{\gamma,y}: (\xi_y,\omega) \longrightarrow (\xi_y,\omega)
\]
called the {\bf linearized return map\/}.  The eigenvalues of
$P_{\gamma,y}$ do not depend on $y$.  The Reeb orbit $\gamma$ is said
to be {\bf nondegenerate\/} if $P_{\gamma,y}$ does not have $1$ as an
eigenvalue.  {\em In this paper we always assume that all Reeb orbits
  are nondegenerate\/}\footnote{As K. Cieliebak pointed out to me, it
  is currently unknown whether an arbitrary stable Hamiltonian
  structure can be slightly perturbed so as to make all Reeb orbits
  nondegenerate.  (However this is not a problem in the contact case
  or the mapping torus case.)}.  For any Reeb orbit $\gamma$, the
linearized return map $P_{\gamma,y}$, being symplectic, has
eigenvalues $\lambda,\lambda^{-1}$ which are either real and positive,
in which case $\gamma$ is called {\bf positive hyperbolic\/}, or real
and negative, in which case $\gamma$ is called {\bf negative
  hyperbolic\/}, or on the unit circle, in which case $\gamma$ is
called {\bf elliptic\/}.

\subsection{The Conley-Zehnder index}

If $\gamma:\R/T\to Y$ is a Reeb orbit, let $\mc{T}(\gamma)$ denote the
set of homotopy classes of symplectic trivializations of the $2$-plane
bundle $\gamma^*\xi$ over $S^1=\R/T$.  This is an affine space over
$\Z$.  Our sign convention\footnote{The paper \cite{pfh2} incorrectly
  claims to be using this convention.  It in fact uses the opposite
  convention throughout.} is that if $\tau_1,\tau_2:\gamma^*\xi\to
S^1\times\R^2$ are two trivializations, then
\begin{equation}
\label{eqn:FSC}
\tau_1-\tau_2=\op{deg}(\tau_2\circ\tau_1^{-1}:S^1\longrightarrow
\op{Sp}(2,\R)\approx S^1).
\end{equation}

Now let $\gamma:\R/T\to Y$ be a Reeb orbit and let $\tau$ be a
trivialization of $\gamma^*\xi$.  Given $t\in\R$, the linearized Reeb
flow along $\gamma$ from time $0$ to time $t$ defines a symplectic map
$\xi_{\gamma(0)}\to\xi_{\gamma(t)}$, which with respect to the
trivialization $\tau$ is a symplectic matrix $\psi(t)$.  In
particular, $\psi(0)$ is the identity and $\psi(T)$ is the linearized
return map.  Since $\gamma$ is assumed nondegenerate, the path of
symplectic matrices $\{\psi(t)\mid 0\le t\le T\}$ has a well-defined
{\bf Conley-Zehnder index\/}, which we denote by
\[
\op{CZ}_\tau(\gamma)\in\Z.
\]
In our three-dimensional situation, this can be described explicitly
as follows.

\begin{itemize}
\item If $\gamma$ is hyperbolic, then there is an integer $n$ such
  that the linearized Reeb flow along $\gamma$ rotates the eigenspaces
  of the linearized return map by angle $n\pi$ with respect to
  $\tau$.  In this case
\begin{equation}
\label{eqn:CZHyp}
\op{CZ}_\tau(\gamma^k) = kn.
\end{equation}
The integer $n$ is even when $\gamma$ is positive hyperbolic and odd
when $\gamma$ is negative hyperbolic.
\item
If $\gamma$ is elliptic, then $\tau$ is homotopic to a trivialization
in which the linearized Reeb flow along $\gamma$ rotates by angle
$2\pi\theta$.  Here the number $\theta$, called the {\bf monodromy
  angle\/}, is necessarily irrational because $\gamma$ and all of its
iterates are assumed nondegenerate.  In this case
\begin{equation}
\label{eqn:CZEll}
\op{CZ}_\tau(\gamma^k) = 2\floor{k\theta}+1.
\end{equation}
\end{itemize}
The Conley-Zehnder index depends only on the Reeb orbit $\gamma$ and
the homotopy class of $\tau$ in $\mc{T}(\gamma)$.  If
$\tau'\in\mc{T}(\gamma)$ is another trivialization, then we have
\begin{equation}
\label{eqn:CZTriv}
\op{CZ}_\tau(\gamma^k) - \op{CZ}_{\tau'}(\gamma^k) = 2k(\tau' - \tau).
\end{equation}

\subsection{Orbit sets}

\begin{definition}
An {\bf orbit set\/} is a finite set of pairs
$\alpha=\{(\alpha_i,m_i)\}$, where:
\begin{itemize}
\item
The $\alpha_i$'s are distinct, embedded Reeb orbits.
\item
The $m_i$'s are positive integers\footnote{Recall that in order to be a
  generator of the ECH chain complex, $\alpha$ must satisfy the
  additional requirement that $m_i=1$ whenever $\alpha_i$ is
  hyperbolic.  However we will not impose that condition anywhere in
  this paper.}.
\end{itemize}
Define the homology
class of $\alpha$ by
\[
[\alpha] \eqdef \sum_i m_i[\alpha_i]\in H_1(Y).
\]
\end{definition}

\begin{definition}
\label{def:RHC}
  If $\alpha=\{(\alpha_i,m_i)\}$ and $\beta=\{(\beta_j,n_j)\}$ are
  orbit sets with $[\alpha]=[\beta]\in H_1(Y)$, let
  $H_2(Y,\alpha,\beta)$ denote the set of relative homology classes of
  $2$-chains $Z$ in $Y$ such that
\[
\partial Z = \sum_i m_i\alpha_i - \sum_j n_j \beta_j.
\]
That is, two such 2-chains represent the same element of
$H_2(Y,\alpha,\beta)$ if and only if their difference is the boundary
of a 3-chain.  Thus $H_2(Y,\alpha,\beta)$ is an affine space over
$H_2(Y)$.
\end{definition}

\subsection{The relative first Chern class}

Fix orbit sets $\alpha=\{(\alpha_i,m_i)\}$ and
$\beta=\{(\beta_j,n_j)\}$ with $[\alpha]=[\beta]\in H_1(Y)$.
Also fix trivializations $\tau_i^+\in\mc{T}(\alpha_i)$ for each
$i$ and
$\tau_j^-\in\mc{T}(\beta_j)$ for each $j$, and denote this set of
trivialization choices by $\tau$.  Let $Z\in H_2(Y,\alpha,\beta)$.

\begin{definition}
\label{def:ctau}
Define the {\bf relative first Chern class\/}
\[
c_\tau(Z) \eqdef c_1(\xi|_Z,\tau)\in\Z
\]
as follows.  Represent $Z$ by a smooth map $f:S\to Y$, where $S$ is a
compact oriented surface with boundary.  Choose a section $\psi$ of
$f^*\xi$ over $S$ such that $\psi$ is transverse to the zero section,
and over each boundary component of $S$, the section $\psi$ is
nonvanishing and has winding number zero with respect to $\tau$.
Define
\[
c_\tau(Z) \eqdef \#\psi^{-1}(0),
\]
where `$\#$' denotes the signed count.
\end{definition}

It is not hard to show that $c_\tau(Z)$ is well defined.  Moreover
\begin{equation}
\label{eqn:cAmb}
c_\tau(Z) - c_\tau(Z') = \langle c_1(\xi),Z-Z'\rangle,
\end{equation}
where $c_1(\xi)\in H^2(Y;\Z)$ denotes the ordinary first Chern class.
Finally, if $\tau'=(\{{\tau_i^+}'\},\{{\tau_j^-}'\})$ is another collection of
trivialization choices, then
\begin{equation}
\label{eqn:cTriv}
c_\tau(Z) - c_{\tau'}(Z) = \sum_i
m_i({\tau_i^+}-{\tau_i^+}') - \sum_j n_j({\tau_j^-}-{\tau_j^-}').
\end{equation}

\subsection{Braids around Reeb orbits}

Let $\gamma$ be an embedded Reeb orbit and let $m$ be a positive
integer.

\begin{definition}
  A {\bf braid\/} around $\gamma$ with $m$ strands is an oriented link
  $\zeta$ contained in a tubular neighborhood $N$ of $\gamma$ such
  that the tubular neighborhood projection $\zeta\to\gamma$ is an
  orientation-preserving degree
  $m$ submersion.
\end{definition}

We now define the writhe, linking number, and winding number of braids
around $\gamma$, which will be used repeatedly below.  For this
purpose choose a trivialization $\tau$ of $\gamma^*\xi$.  Extend the
trivialization to identify the tubular neighborhood $N$ with
$S^1\times D^2$, so that the projection of $\zeta$ to $S^1$ is a
submersion.  Identify $S^1\times D^2$ with a solid torus in $\R^3$ via
the orientation-preserving diffeomorphism sending
\[
(\theta,(x,y)) \longmapsto
(1+x/2)(\cos\theta,\sin\theta,0)-(0,0,y/2).
\]
This defines an embedding $\phi_\tau:N\to\R^3$.  Now $\phi_\tau(\zeta)$
is an oriented link in $\R^3$ with no vertical tangents.  As such, it
has a well-defined writhe, which is the signed count of the crossings
in the projection to $\R^2\times\{0\}$, after perturbing the link to
have generic crossings.  The sign convention is that counterclockwise
twists contribute positively to the writhe.

\begin{definition}
If $\zeta$ is a braid around $\gamma$, define the {\bf writhe\/}
\[
w_\tau(\zeta)\in\Z
\]
to be the writhe of the oriented link $\phi_\tau(\zeta)$ in $\R^3$.
\end{definition}

This depends only on the isotopy class of
$\zeta$ and the homotopy class of $\tau$ in $\mc{T}(\gamma)$.  If
$\tau'\in\mc{T}(\gamma)$ is another trivialization, and if $\zeta$ has
$m$ strands, then
\begin{equation}
\label{eqn:wTriv}
w_\tau(\zeta) - w_{\tau'}(\zeta) = m(m-1)(\tau' - \tau),
\end{equation}
because shifting the trivialization by one adds a full clockwise twist
to the braid $\phi_\tau(\zeta)$.

\begin{definition}
If $\zeta_1$ and $\zeta_2$ are disjoint braids around $\gamma$, define
the {\bf linking number\/}
\[
\ell_\tau(\zeta_1,\zeta_2) \in \Z
\]
to be the linking number of the oriented links $\phi_\tau(\zeta_1)$
and $\phi_\tau(\zeta_2)$ in $\R^3$.  The latter is, by definition, one
half the signed count of crossings of a strand of $\phi_\tau(\zeta_1)$
with a strand of $\phi_\tau(\zeta_2)$ in the projection to
$\R^2\times\{0\}$.
\end{definition}

Similarly to \eqref{eqn:wTriv}, if $\zeta_k$ has $m_k$ strands, then
\begin{equation}
\label{eqn:ltriv}
\ell_\tau(\zeta_1,\zeta_2) - \ell_{\tau'}(\zeta_1,\zeta_2) =
m_1m_2(\tau'-\tau).
\end{equation}
Also note that
\begin{equation}
\label{eqn:writheUnion}
w_\tau(\zeta_1\cup\zeta_2) = w_\tau(\zeta_1) + w_\tau(\zeta_2) + 2
\ell_\tau(\zeta_1,\zeta_2).
\end{equation}

\begin{definition}
If $\zeta$ is a braid around $\gamma$ which is disjoint from $\gamma$,
define the {\bf winding number\/}
\[
\eta_\tau(\zeta) \eqdef \ell_\tau(\zeta,\gamma)\in\Z.
\]
\end{definition}

\subsection{The relative intersection pairing}
\label{sec:Q}

\begin{definition}
Let $\alpha=\{(\alpha_i,m_i)\}$ and $\beta=\{(\beta_j,n_j)\}$ be orbit
sets with $[\alpha]=[\beta]$, and let $Z\in H_2(Y,\alpha,\beta)$.
An {\bf admissible representative\/} of $Z$ is a smooth map $f:S\to
[-1,1]\times Y$, where $S$ is a compact oriented surface with
boundary, such that:
\begin{itemize}
\item
The restriction of $f$ to $\partial S$ consists of positively oriented
covers of $\{1\}\times \alpha_i$ with total multiplicity $m_i$ and
negatively oriented covers of $\{-1\}\times\beta_j$ with total
multiplicity $n_j$.
\item The composition of $f$ with the projection $[-1,1]\times
  Y\to Y$ represents the class $Z$.
\item The restriction of $f$ to the interior of $S$ is an embedding,
  and $f$ is transverse to $\{-1,1\}\times Y$.
\end{itemize}
We will generally abuse notation and denote the admissible
representative by $S$.  It is not hard to see that any class $Z$ has
an admissible representative; we will construct some special admissible
representatives in \S\ref{sec:Qweta} below.
\end{definition}

If $S$ is an admissible representative of $Z$, then for $\epsilon>0$
sufficiently small, $S\cap (\{1-\epsilon\}\times Y)$ consists of
braids $\zeta_i^+$ with $m_i$ strands in disjoint tubular
neighborhoods of the Reeb orbits $\alpha_i$, which are well defined up
to isotopy.  Likewise $S\cap (\{-1+\epsilon\}\times Y)$ consists of
disjoint braids $\zeta_j^-$ with $n_j$ strands in disjoint tubular
neighborhoods of the Reeb orbits $\beta_j$.  If $S'$ is an admissible
representative of $Z'\in H_2(Y,\alpha',\beta')$, such that the
interior of $S'$ does not intersect the interior of $S$ near the
boundary, with braids ${\zeta_i^+}'$ and ${\zeta_j^-}'$, define the
linking number
\[
\ell_\tau(S,S') \eqdef \sum_i \ell_\tau(\zeta_i^+,{\zeta_i^+}')
- \sum_j \ell_\tau(\zeta_j^-,{\zeta_j^-}').
\]
Here we are using the same index $i$ for the orbit sets $\alpha$ and
$\alpha'$, so that sometimes $m_i=0$ or $m_i'=0$, and likewise the same index
$j$ for the orbit sets $\beta$ and $\beta'$; and $\tau$ is a
trivialization of $\xi$ over all Reeb orbits in $\alpha$, $\alpha'$,
$\beta$, and $\beta'$.

\begin{definition}
\label{def:Q}
If $Z\in H_2(Y,\alpha,\beta)$ and $Z'\in H_2(Y,\alpha',\beta')$,
define the {\bf relative intersection number\/}
\[
Q_\tau(Z,Z')\in\Z
\]
as follows.  Choose admissible representatives $S$ of $Z$ and $S'$ of
$Z'$ whose interiors $\dot{S}$ and $\dot{S}'$ are transverse and do
not intersect near the boundary.  Define
\[
Q_\tau(Z,Z') \eqdef \#(\dot{S}\cap \dot{S}') - \ell_\tau(S,S').
\]
\end{definition}

It follows from \cite[Lemmas 2.5 and 8.5]{pfh2} that this is well
defined, and moreover
\begin{equation}
\label{eqn:QAmb}
Q_\tau(Z_1,Z') - Q_\tau(Z_2,Z') = (Z_1-Z_2)\cdot[\alpha'],
\end{equation}
where `$\cdot$' denotes the ordinary intersection number in $Y$.
Clearly $Q_\tau$ is symmetric: $Q_\tau(Z,Z')=Q_\tau(Z',Z)$.  Also, it
follows from \eqref{eqn:ltriv} that if
$\tau'=(\{{\tau_i^+}'\},\{{\tau_j^-}'\})$ is another collection of
trivialization choices, then
\begin{equation}
\label{eqn:QTriv}
Q_\tau(Z,Z') - Q_{\tau'}(Z,Z') = \sum_i m_im_i'({\tau_i^+}-{\tau_i^+}') -
\sum_j n_jn_j'({\tau_j^-} - {\tau_j^-}').
\end{equation}
The most important case is where $Z=Z'$; we denote this by
\[
Q_\tau(Z) \eqdef Q_\tau(Z,Z).
\]

\subsection{Definition of the ECH index}
\label{sec:DECHI}

\begin{notation}
If $\alpha=\{(\alpha_i,m_i)\}$ is an orbit set and $\tau=\{\tau_i\}$ is a
trivialization of $\xi$ over the $\alpha_i$'s, define
\[
\mu_\tau(\alpha) \eqdef \sum_i \sum_{k=1}^{m_i}
\op{CZ}_{\tau_i}(\alpha_i^k).
\]
By equation \eqref{eqn:CZTriv}, if $\tau'=\{\tau_i'\}$ is another set
of trivialization choices, then
\begin{equation}
\label{eqn:muTriv}
\mu_\tau(\alpha) - \mu_{\tau'}(\alpha) = \sum_i(m_i^2+m_i)(\tau_i'-\tau_i).
\end{equation}
\end{notation}

\begin{definition}
  Let $\alpha$ and $\beta$ be orbit sets with
  $[\alpha]=[\beta]=[\Gamma]\in H_1(Y)$, and let $Z\in
  H_2(Y,\alpha,\beta)$.  Define the {\bf ECH index\/}
\[
I(\alpha,\beta,Z) \eqdef c_\tau(Z) + Q_\tau(Z) +
\mu_\tau(\alpha) - \mu_\tau(\beta).
\]
Here $\tau$ is a trivialization of $\xi$ over the $\alpha_i$'s and
$\beta_j$'s.  It follows from equations \eqref{eqn:cTriv},
\eqref{eqn:QTriv}, and \eqref{eqn:muTriv} that $I$ does not depend on
$\tau$.
\end{definition}

We can also define a version of $I$ does not depend on a class $Z$.
Namely, by equations \eqref{eqn:cAmb} and
\eqref{eqn:QAmb}, if $Z'\in H_1(Y,\alpha,\beta)$, then we have the
{\em index ambiguity formula\/}
\[
I(\alpha,\beta,Z) - I(\alpha,\beta,Z') = \langle c_1(\xi) + 2
\op{PD}(\Gamma), Z-Z'\rangle.
\]
Here $\op{PD}(\Gamma)\in H^2(Y;\Z)$ denotes the Poincare dual of
$\Gamma$.  Thus the following definition makes sense:

\begin{definition}
If $\alpha$ and $\beta$ are orbit sets with $[\alpha]=[\beta]=\Gamma$,
define
\begin{equation}
\label{eqn:Ialphabeta}
I(\alpha,\beta) \eqdef I(\alpha,\beta,Z) \in \Z/d(c_1(\xi) + 2
\op{PD}(\Gamma))
\end{equation}
where $Z$ is any class in $H_2(Y,\alpha,\beta)$, and $d$ denotes
divisibility in $H^2(Y;\Z)$ modulo torsion.
\end{definition}

\begin{remark}
  It is easy to show, see \cite{pfh2}, that $I$ is {\em additive\/} in
  the following sense: if $\gamma$ is another orbit set with
  $[\gamma]=\Gamma$, and if $W\in H_2(Y,\beta,\gamma)$, then
\[
I(\alpha,\beta,Z) + I(\beta,\gamma,W) = I(\alpha,\gamma,Z+W).
\]
Thus $I$ defines a relative grading on ECH
generators.
\end{remark}

\section{An absolute ECH index}
\label{sec:AG}

We now explain how to refine the relative index $I(\alpha,\beta)$ in
\eqref{eqn:Ialphabeta} to an absolute index, which associates to each
orbit set a homotopy class of oriented 2-plane fields on $Y$.

\subsection{Homotopy classes of oriented $2$-plane fields}
\label{sec:TP}

Before stating the result, we briefly recall some basic facts about
homotopy classes of oriented $2$-plane fields which we will need.  For
proofs of the less obvious of these facts see e.g.\ \cite[\S4]{gompf}
and \cite[Ch.\ 28]{km}.

Let $Y$ be a connected, closed oriented $3$-manifold.  Let $\mc{P}(Y)$
denote the set of homotopy classes of oriented $2$-plane fields on
$Y$.  This is the same as the set of homotopy classes of nonvanishing
vector fields on $Y$.  Let $\Spinc(Y)$ denote the set of spin-c
structures on $Y$; this is an affine space over $H^2(Y;\Z)$.  There is
a surjection
\[
\frak{s}: \mc{P}(Y) \longrightarrow \Spinc(Y).
\]
Given two oriented $2$-plane fields $\xi_1$ and $\xi_2$, the primary
obstruction to finding a homotopy between them is the difference
between the corresponding spin-c structures,
\begin{equation}
\label{eqn:ob2}
\frak{s}(\xi_1) - \frak{s}(\xi_2) \in H^2(Y;\Z).
\end{equation}
Note that
\[
c_1(\xi_1) - c_1(\xi_2) = 2(\frak{s}(\xi_1) - \frak{s}(\xi_2)).
\]
In particular, if $\xi_1$
and $\xi_2$ determine the same spin-c structure, then
$c_1(\xi_1)=c_1(\xi_2)$, and the secondary
obstruction to finding a homotopy between them is a class
\begin{equation}
\label{eqn:ob3}
[\xi_1] - [\xi_2] \in \Z/d(c_1(\xi_1)).
\end{equation}
Thus the set of homotopy classes of $2$-plane fields determining this
spin-c structure is an affine space over $\Z/d(c_1(\xi_1))$.  Our sign
convention for the affine structure is specified by the isomorphism
$\pi_3(S^2)\simeq\Z$ that identifies the Hopf fibration with $+1$.

It will be useful below to understand the obstructions \eqref{eqn:ob2}
and \eqref{eqn:ob3} in terms of Thom-Pontrjagin theory as follows.
Let $\mc{L}(Y)$ denote the set of oriented framed links in $Y$, modulo
framed link cobordism.  We have a surjection $\mc{L}(Y)\to H_1(Y)$
sending a link $L$ to its homology class $[L]\in H_1(Y)$.  There is
also a $\Z$-action on $\mc{L}(Y)$ by twisting the framings; on the set
of elements of $\mc{L}(Y)$ with homology class $\Gamma\in H_1(Y)$, the
stabilizer of this $\Z$ action is $2d(\Gamma)$. Our sign convention
for this $\Z$-action is given as in \eqref{eqn:FSC}, but with
$\op{Sp}(2,\R)$ replaced by $\op{GL}^+(2,\R)$.
Now fix a
trivialization $\rho$ of $TY$.  Then an oriented $2$-plane field,
regarded as a nonvanishing vector field, defines a map $Y\to S^2$, and
the inverse image of a regular value of this map is an oriented framed
link.  This construction defines a bjiection
\[
L_\rho: \mc{P}(Y) \longrightarrow \mc{L}(Y)
\]
satisfying $2 [L_\rho(\xi)] = \op{PD}(c_1(\xi))$.  In terms of this
correspondence, the two-dimensional obstruction
\eqref{eqn:ob2} is given by
\begin{equation}
\label{eqn:TP2}
\frak{s}(\xi_1) - \frak{s}(\xi_2) = \op{PD}([L_\rho(\xi_1)] -
[L_\rho(\xi_2)]).
\end{equation}
The three-dimensional obstruction \eqref{eqn:ob3} is described as
follows: if $\xi_1$ and $\xi_2$ determine the same spin-c structure,
then
\begin{equation}
\label{eqn:TP3}
[\xi_1] - [\xi_2] = L_\rho(\xi_1) - L_\rho(\xi_2).
\end{equation}
This last equation means that the framed link cobordism classes
$L_\rho(\xi_1)$ and $L_\rho(\xi_2)$ can be represented by the same
link, but with the framings differing by $[\xi_1]-[\xi_2]$.

Now suppose we allow our compact connected oriented 3-manifold $Y$ to
have boundary.  Let $\xi_0$ be an oriented rank $2$ subbundle of
$TY|_{\partial Y}$.  Define $\mc{P}(Y,\xi_0)$ to be the set of
homotopy classes of oriented $2$-plane fields on $Y$ that restrict to
$\xi_0$ on $\partial Y$.  (These correspond to spin-c structures
$\frak{s}$ on $Y$ together with an isomorphism $\frak{s}|_{\partial
  Y}\simeq\frak{s}_0$ where $\frak{s}_0$ is a fixed spin-c structure
on $\partial Y$ determined by $\xi_0$.)  Given two elements
$\xi_1,\xi_2\in\mc{P}(Y,\xi_0)$, the primary obstruction to finding a
homotopy between them is an element of $H^2(Y,\partial Y;\Z)$.  The
image of this obstruction in $H^2(Y;\Z)$, multiplied by $2$, equals
$c_1(\xi_1)-c_1(\xi_2)$.  If the primary obstruction vanishes, then
the secondary obstruction is an element of $\Z/d(c_1(\xi_1))$.  To
describe these obstructions in terms of Thom-Pontrjagin theory, choose
a trivialization $\rho$ of $TY$.  This gives rise to a
zero-dimensional nullhomologous oriented framed submanifold
$F\subset\partial Y$.  Let $\mc{L}(Y,F)$ denote the set of
oriented framed links on $Y$ with boundary $F$, modulo framed
cobordism relative to $F$.  Then as before we have a bijection
\[
L_\rho: \mc{P}(Y,\xi_0) \longrightarrow \mc{L}(Y,F).
\]
Given $\xi_1,\xi_2\in\mc{P}(Y,\xi_0)$, the primary obstruction to
finding a homotopy between them is Poincar\'{e} dual to the difference
in relative homology classes $[L_\rho(\xi_1)]-[L_\rho(\xi_2)]\in
H_1(Y)$ as in \eqref{eqn:TP2}, and if this vanishes then the secondary
obstruction is the difference in framings as in \eqref{eqn:TP3}.

\subsection{Statement of the result}

Fix a connected closed oriented 3-manifold $Y$ with a stable Hamiltonian structure
such that all Reeb orbits are nondegenerate.

\begin{theorem}
\label{thm:AG}
For each orbit set $\alpha=\{(\alpha_i,m_i)\}$, there is a
homotopy class of oriented $2$-plane fields $I(\alpha)\in\mc{P}(Y)$,
such that:
\begin{description}
\item{(a)} $I(\alpha)$ is obtained by modifying $\xi$ in a canonical
  manner (up to homotopy, depending only on $m_i$) in disjoint tubular
  neighborhoods of each $\alpha_i$.
\item{(b)} $\frak{s}(I(\alpha)) = \frak{s}(\xi) + \op{PD}([\alpha])$.
\item{(c)} If $\alpha$ and $\beta$ are orbit sets with
  $[\alpha]=[\beta]=\Gamma$, then
\begin{equation}
\label{eqn:IRA}
I(\alpha,\beta) = I(\alpha) - I(\beta)
\end{equation}
in $\Z/d(c_1(\xi) + 2
\op{PD}(\Gamma))$.
\end{description}
\end{theorem}

\begin{remark}
  Part (c) asserts that $I$ defines an absolute grading on ECH
  which refines the relative grading in \eqref{eqn:Ialphabeta}.
  Part (a) implies that the absolute grading $I$ of the empty
  set\footnote{If $Y$ is a contact manifold, then the empty set is a
    very important ECH generator, which is a cycle in the ECH chain
    complex, whose homology class in ECH conjecturally agrees with the
    contact invariants in the Seiberg-Witten and Heegaard Floer
    homologies.} is the homotopy class of the 2-plane field $\xi$
  itself.  Part (b) asserts that $I(\alpha)$ determines the
  correct spin-c structure, so that it makes sense to conjecture that
  Taubes's isomorphism \eqref{eqn:echswf} between ECH and
  Seiberg-Witten Floer homology respects the absolute gradings.
\end{remark}

\subsection{Modifying the $2$-plane field near transversal
  links}

To prepare for the proof of the theorem, consider a transversal link
$L\subset Y$.  This means that $L$ is transverse to the $2$-plane
field $\xi$ at every point.  As such, $L$ has a canonical
orientation.  Now let $\tau$ be a framing of $L$, i.e.\ a homotopy
class of symplectic trivialization of $\xi|_L$.

\begin{definition}
\label{def:PtauL}
  Given a transversal link $L$ with framing $\tau$, define a homotopy
  class of oriented $2$-plane fields $P_\tau(L)$ as follows.

  Let $N$ be a tubular neighborhood of $L$. On $Y\setminus N$,
  take $P_\tau(L)\eqdef \xi$.

  To describe $P\eqdef P_\tau(L)$ on $N$, for each component $K$ of
  $L$, let $N_K$ denote the corresponding component of $N$. Choose a
  diffeomorphism
\[
\phi_K: N_K \stackrel{\simeq}{\longrightarrow} S^1\times D^2
\]
such that $\phi_K$ sends $K$ to $S^1\times\{0\}$, and the derivative
$d\phi_K$ sends $\xi|_{K}$ to $\{0\}\oplus\R^2$, compatibly with the
framing $\tau$.  Extend the latter to a trivialization of $TN_K$,
identifying $\xi=\{0\}\oplus\R^2$ and $R=(1,0,0)$ at each point.
On $N_K$, choose $P$, regarded as a vector field, so that:
\begin{itemize}
\item 
On $S^1\times\{z\in D^2 \mid |z|>1/2\}$, the vector field $P$
intersects $\xi$ positively.
\item
On $S^1\times\{z\in D^2 \mid |z|<1/2\}$ the vector field $P$
intersects $\xi$ negatively.
\item On $S^1\times\{z\in D^2 \mid |z|=1/2\}$, the vector field $P$,
  regarded using the above trivialization as a function with values in
  $\R\oplus\R^2$, is given by
\begin{equation}
\label{eqn:P12}
P(t,e^{i\theta}/2) \eqdef (0,e^{-i\theta}).
\end{equation}
\end{itemize}
These conditions uniquely determine $P_\tau(L)$ up to homotopy.
\end{definition}

The following are some basic properties of $P_\tau(L)$.

\begin{lemma}
\label{lem:PtauL}
\begin{description}
\item{(a)}
$\frak{s}(P_\tau(L)) = \frak{s}(\xi) + \op{PD}([L])$.
\item{(b)}
If $\tau'$ is a different framing of $L$, then
\[
P_\tau(L) - P_{\tau'}(L) \equiv 2(\tau-\tau') \mod d(c_1(\xi) + 2
\op{PD}([L])).
\]
\item{(c)} If $L_\pm$ is obtained from $L$ by locally fusing two
  strands into a crossing of sign $\pm 1$ as shown in
  Figure~\ref{fig:skein}, and if the framings $\tau$ of $L_\pm$ and
  $L$ are the blackboard framing in Figure~\ref{fig:skein} and agree
  everywhere else, then
\[
P_\tau(L_\pm) - P_\tau(L) \equiv \pm 1 \mod d(c_1(\xi) + 2\op{PD}([L]).
\]
\item{(d)} Let $S\subset Y$ be an embedded compact oriented surface
  with $\partial S = \widehat{L}_1 \sqcup -\widehat{L}_2$ where
  $\widehat{L}_1$ and $\widehat{L}_2$ are transversal links.  Let
  $L_0$ be a transversal link disjoint from $S$, let $L_1\eqdef
  \widehat{L}_1\sqcup L_0$, and $L_2\eqdef \widehat{L}_2\sqcup L_0$.
  Then
\[
P_\tau(L_1) - P_\tau(L_2) \equiv c_1(\xi|_S,\tau) \mod
d(c_1(\xi)+2\op{PD}([L_1])),
\]
where the framings $\tau$ are induced from the conormal direction to
$S$ on $\widehat{L}_1$ and $\widehat{L}_2$ and
an arbitrary framing on $L_0$.
\end{description}
\end{lemma}

\begin{figure}
\begin{center}
\scalebox{0.6}{\includegraphics{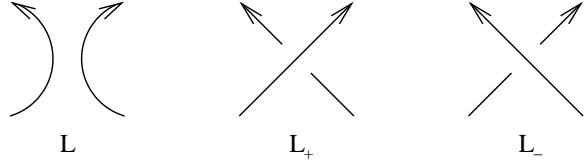}}
\end{center}
\caption{The links $L$, $L_+$, and $L_-$ in Lemma~\ref{lem:PtauL}(c).}
\label{fig:skein}
\end{figure}

\begin{proof}
We will prove all four assertions using the Thom-Pontrjagin theory
from \S\ref{sec:TP}.

To start, in the definition of $P_\tau(L)$, we can then take the
vector field $P$ on $N_K$, regarded as a function $S^1\times
D^2\to\R\oplus\R^2$, to be
\begin{equation}
\label{eqn:explicitP}
P(t,re^{i\theta}) = (-\cos(\pi r), \sin(\pi r)e^{-i\theta}).
\end{equation}
Then $(-1,0)$ is a regular value of $P$, whose inverse image is the
core circle $S^1\times\{0\} \subset S^1\times D^2$.  This circle is
oriented positively.  In terms of the Thom-Pontrjagin construction,
this means that on $N_K$,
\[
[L_\rho(P)] - [L_\rho(\xi)] = [K] \in H_1(N_K).
\]
Together with \eqref{eqn:TP2}, this implies assertion (a).

To calculate the framing on $L_\rho(P)$ above, we can take a nearby
regular value of $P$ in $S^2$ such as $(0,1,0)$.  The inverse image of
this is $S^1\times\{(1/2,0)\}$, and so $L_\rho(P)$ has framing $0$
with respect to our trivializations.  Now let $\tau'$ be another
framing of $L$ such that on $K$ we have $\tau-\tau'=k\in \Z$.  The
difference in trivializations
\[
\Phi\eqdef \tau' \circ \tau^{-1}: S^1 \longrightarrow \op{Sp}(2,\R)
\]
can be taken to be
\[
\Phi(t) = e^{ikt}.
\]
This is induced by a diffeomorphism $\widetilde{\Phi}:S^1\times D^2\to
S^1\times D^2$ given by
\[
\widetilde{\Phi}(t,re^{i\theta}) = (t,re^{i(\theta+kt)}).
\]
With respect to the previous trivialization $\rho$ of $TN_K$ coming
from $\tau$, a vector field $P'$ corresponding to $\tau'$ is given by
the function
\[
P' = (1 \oplus \Phi^{-1})\circ P\circ \widetilde{\Phi},
\]
where $P$ is the function defined by \eqref{eqn:explicitP}.  This
comes out to be
\[
P'(t,re^{i\theta}) = (-\cos(\pi r), \sin(\pi r)e^{-i(\theta + 2kt)}).
\] 
Again, the link $L_\rho(P')$ on $S^1\times S^2$ is the circle
$S^1\times\{0\}$, oriented positively.  However now the regular value
$(0,1,0)$ of the map $P'$ to $S^2$ has inverse image
$\{(t,\frac{1}{2}e^{-2ikt})\}$, so $L_\rho(P')$ has framing $-2k$.
Together with \eqref{eqn:TP3}, this implies assertion (b) of the
lemma.

Assertion (c) follows immediately from the Thom-Pontrjagin
construction.

We now prove assertion (d).  To start, we may assume that $S$ has no
closed components.  For if $S_0$ is the union of the closed components
of $S$, then
\[
c_1(\xi|_{S_0},\tau) = \langle c_1(\xi),[S_0]\rangle = \langle
c_1(\xi) + 2\op{PD}([L_1]),[S_0]\rangle,
\]
since $S_0$ is disjoint from $L_1$.  Thus removing $S_0$ from $S$ does
not affect the validity of the congruence that we need to prove.

We may then also assume that $S$ is connected, since tubing together
different components of $S$ has no effect on $c_1(\xi|_S,\tau)$.

Now let $N\subset Y$ be a neighborhood of $S$, identified with
$\widetilde{S}\times[-1,1]$, where $\widetilde{S}\supset S$ is
obtained by extending $S$ slightly past its boundary, so that the
identification $N\simeq \widetilde{S}\times[-1,1]$ sends $S$ to
$S\times\{0\}$.  Choose $N$ small enough so that it is disjoint from
$L_0$.  It is enough to show that in $\mc{P}(N,\xi|_{\partial N})$ we
have
\begin{equation}
\label{eqn:PGoal}
P_\tau(\widehat{L}_1) - P_\tau(\widehat{L}_2) = c_1(\xi|_S,\tau) \in \Z.
\end{equation}

Since $S$ has nonempty boundary, we can choose a trivialization
$\rho:TN\to N\times \R^3$ identifying $TS$ with
$S\times(\R^2\oplus\{0\})$.  We can choose this trivialization so that
$(0,0,\pm 1)$ is a regular value of the map $S\to S^2$ given by the
normalized Reeb vector field; let $T_\pm\subset S$ denote the inverse
image of $(0,0,\pm 1)$ under this map, and let $t_\pm\in\Z$ denote the
signed count of points in the set $T_\pm$.

For $i=1,2$, we can take $L_\rho(P_\tau(\widehat{L}_i))$ to be the
set of points in $N$ such that the vector field corresponding to
$P_\tau(\widehat{L}_i)$ points in the direction $(0,0,1)$.  If this
vector field is chosen appropriately, then the framed
link $L_\rho(P_\tau(\widehat{L}_i))$ consists of the following:
\begin{itemize}
\item A vertical line segment at each point in $T_+$, such that the
  number of upward pointing segments minus the number of downward
  point segments equals $t_+$.
\item A vertical pushoff of the link $\widehat{L}_i$, with the
  conormal framing.
\end{itemize}
It follows using assertion (c) and equation \eqref{eqn:TP3} that
\[
P_\tau(\widehat{L}_1) - P_\tau(\widehat{L}_2) = 2t_+ - \chi(S).
\]
Applying the Poincare-Hopf index theorem to the projection of $R$ onto
$TS$ gives
\[
t_+ - t_- = \chi(S).
\]
On the other hand, the projection of $(0,0,1)$ to $\xi$ defines a
section of $\xi|_S$ which is zero exactly where the Reeb vector field is
vertical, showing that
\[
t_+ + t_- = c_1(\xi|_S,\tau),
\]
compare \cite[\S4.2]{etnyre}. Combining the above three equations
proves \eqref{eqn:PGoal}.
\end{proof}

\begin{remark}
\label{rem:P'}
  One can define another homotopy class of oriented $2$-plane fields
  $P'(L)$, following Definition~\ref{def:PtauL}, but with equation
  \eqref{eqn:P12} replaced by
\[
P'(t,e^{i\theta}/2) \eqdef (0,e^{i\theta}).
\]
The homotopy class of oriented $2$-plane fields $P'(L)$ satisfies
$\frak{s}(P'(L)) = \frak{s}(\xi) - \op{PD}([L])$, does not depend on a
framing of $L$, satisfies the analogue of property (c) above, and the
analogue of property (d) but with the opposite sign.  When $\xi$ is a
contact structure, the homotopy class $P'(L)$ corresponds to the
contact structure obtained from $\xi$ by a {\em Lutz twist\/} along
$L$, see e.g.\ \cite{geiges}.  Although $P'(L)$ is not relevant for
Theorem~\ref{thm:AG}, it is significant in connection with defining a
relative filtration on ECH, see Proposition~\ref{prop:JBasics}.
\end{remark}

\subsection{Definition of the absolute grading}

\begin{definition}
Given an orbit set $\alpha=\{(\alpha_i,m_i)\}$, define a homotopy
class of oriented $2$-plane fields $I(\alpha)\in\mc{P}(Y)$ as
follows.  Choose trivializations $\tau=\{\tau_i\}$ of $\xi$ over the
$\alpha_i$'s.  For each $i$, choose a braid $\zeta_i$ around
$\alpha_i$ with $m_i$ strands.  Assume that the $\zeta_i$'s are in
disjoint tubular neighborhoods of the $\alpha_i$'s.  Consider the
transverse link
$L\eqdef\bigcup_i\zeta_i$, with the framing $\tau$ induced by the
$\tau_i$'s.  Define
\begin{equation}
\label{eqn:DAG}
I(\alpha) \eqdef P_\tau(L) - \sum_iw_{\tau_i}(\zeta_i) +
  \mu_\tau(\alpha).
\end{equation}
\end{definition}

\begin{lemma}
\label{lem:IWD}
$I(\alpha)$ is well-defined.
\end{lemma}

\begin{proof}
First fix the trivialization choices and consider replacing the braids
$\zeta_i$ by some other braids $\zeta_i'$.  Then by
Lemma~\ref{lem:PtauL}(c), we have
\[
P_\tau(L) - P_\tau(L') = \sum_i\left(w_{\tau_i}(\zeta_i) -
  w_{\tau_i}(\zeta_i')\right).
\]
Thus for given trivializations, $I(\alpha)$ does not depend on the
choice of braids.

Now fix the braids and consider a different set of trivialization choices
$\tau'=\{\tau_i'\}$.  Changing the trivialization over $\alpha_i$ from
$\tau_i'$ to $\tau_i$ shifts the induced framing on $\zeta_i$ by
$m_i(\tau_i-\tau_i')$.  Thus by Lemma~\ref{lem:PtauL},
\[
P_\tau(L) - P_{\tau'}(L) = \sum_i2m_i(\tau_i-\tau_i').
\]
Combining this with equations \eqref{eqn:wTriv} and \eqref{eqn:muTriv}
proves that $I(\alpha)$ does not depend on the trivialization choices.
\end{proof}

We now want to prove that $I(\alpha)$ satisfies properties (a), (b),
and (c) in Theorem~\ref{thm:AG}.  Property (a) is clear from the proof
of Lemma~\ref{lem:IWD}.  Property (b) is immediate from
Lemma~\ref{lem:PtauL}(a).

\subsection{Computing $Q$ using embedded surfaces in $Y$}
\label{sec:Qweta}

To prepare for the proof of Theorem~\ref{thm:AG}(c), we now establish
a general formula for the relative intersection pairing $Q$ in terms
of embedded surfaces in $Y$.  We use the notation from \S\ref{sec:Q}.

\begin{definition}
An admissible representative $S$ of a class $Z\in H_2(Y,\alpha,\beta)$
is {\bf nice\/} if the projection of $S$ to $Y$ is an immersion, and
the projection of the interior $\dot{S}$ to $Y$ is an embedding which
does not intersect the $\alpha_i$'s or $\beta_j$'s.
\end{definition}

\begin{lemma}
\label{lem:nice}
If none of the $\alpha_i$'s equals any of the $\beta_j$'s, then
every class $Z\in H_2(Y,\alpha,\beta)$ has a nice representative.
\end{lemma}

\begin{proof}
Let $N$ be the union of disjoint tubular neighborhoods of
the $\alpha_i$'s and $\beta_j$'s.  Then $Z$ determines a relative
homology class in
\[
H_2(Y\setminus N, \partial N) = H^1(Y\setminus N;\Z)=[Y\setminus N, S^1].
\]
The latter can be represented by an embedded oriented surface
$S_0\subset Y\setminus N$ transverse to $\partial N$. On each
component of $\partial N$, one can successively cap off contractible
circles in $S_0\cap \partial N$, then cancel adjacent parallel arcs
with opposite orientations, then straighten the remaining arcs, to
arrange that $S_0\cap \partial N$ is a union of torus braids around
each $\alpha_i$ with $m_i$ strands intersecting $\xi$ positively, and
around each $\beta_j$ with $n_j$ strands intersecting $\xi$
negatively.  (These braids have the correct number of strands
because of our assumption that none of the $\alpha_i$'s equals any of
the $\beta_j$'s.)  We can now fill
in $S_0$ over $N$ and lift it to $\R\times Y$ to obtain the desired
nice representative.
\end{proof}

If $S$ is a nice representative of $Z$ with associated braids
$\zeta_i^+$ and $\zeta_j^-$, then it makes sense to define the winding number
\[
\eta_\tau(S) \eqdef \sum_i \eta_{\tau_i^+}(\zeta_i^+) - \sum_j
\eta_{\tau_j^-}(\zeta_j^-).
\]
Also, if $S$ is any admissible representative of $Z$, 
define the writhe
\begin{equation}
\label{eqn:wtauS}
w_\tau(S) \eqdef \sum_i w_{\tau_i^+}(\zeta_i^+) - \sum_j
w_{\tau_j^-}(\zeta_j^-).
\end{equation}

\begin{lemma}
\label{lem:Qweta}
Suppose that $S$ is a nice representative of $Z$.  Then
\[
Q_\tau(Z) = -w_\tau(S) - \eta_\tau(S).
\]
\end{lemma}

\begin{proof}
  Choose a smooth function $\varphi:[-1,1]\to[-1,1]$ such that
  $\varphi(s)\ge s$, with equality only for $s\in\{\pm1\}$. Make
  another admissible representative $S'$ of $Z$ by composing $S$ with
  the diffeomorphism from $[-1,1]\times Y$ to itself that sends
  $(s,y)\mapsto (\varphi(s),y)$.  The corresponding braid
  ${\zeta_i^+}'$ is obtained by pushing $\zeta_i^+$ radially towards
  $\alpha_i$, so their linking number is given by
\[
\ell_{\tau_i^+}(\zeta_i^+,{\zeta_i^+}') = w_{\tau_i^+}(\zeta_i^+) +
\eta_{\tau_i^+}(\zeta_i^+).
\]
Combining this with an analogus formula for the negative braids, we
obtain
\[
\ell_\tau(S,S') = w_\tau(S) + \eta_\tau(S).
\]
On the other hand, since the projection of $S$ to $Y$ is an embedding
on the interior, it follows that $S$ does not intersect $S'$ on the
interior, so
\[
\#(\dot{S}\cap \dot{S}')=0.
\]
The lemma now follows from the definition of $Q$.
\end{proof}

If $\alpha$ and $\beta$ have some Reeb orbits in common, then a nice
representative of $Z$ might not exist, but the above formula for
$Q_\tau(Z)$ can be extended to this case as follows.

To start, it follows from the definition that $Q_\tau$ is quadratic in
the following sense: if $Z\in H_2(Y,\alpha,\beta)$ and $Z'\in
H_2(Y,\alpha',\beta')$, then
\[
Z+Z'\in
H_2(Y,\alpha\alpha',\beta\beta')
\]
is defined (here the product of two
orbit sets is defined by adding the multiplicities of all Reeb orbits
involved), and
\begin{equation}
\label{eqn:quadratic}
Q_\tau(Z+Z') = Q_\tau(Z) + 2Q_\tau(Z,Z') + Q_\tau(Z').
\end{equation}
Here $\tau$ is a trivialization of $\xi$ over all Reeb orbits under
consideration.

Now let $\widehat{\alpha}$ and $\widehat{\beta}$ be obtained from
$\alpha$ and $\beta$ by ``dividing by their greatest common factor''
according to the following procedure: Whenever $\alpha_i=\beta_j$,
replace $m_i$ by $\widehat{m}_i\eqdef m_i - \op{min}(m_i,n_j)$ and
replace $n_j$ by $\widehat{n}_j \eqdef n_j - \op{min}(m_i,n_j)$; then
discard all pairs $(\alpha_i,\widehat{m}_i)$ with $\widehat{m}_i=0$
and $(\beta_j,\widehat{n}_j)$ with $\widehat{n}_j=0$.  Now
$\widehat{\alpha}$ and $\widehat{\beta}$ have no Reeb orbits in
common.  Let $\gamma$ denote the ``greatest common factor'' of
$\alpha$ and $\beta$, namely
\[
\gamma\eqdef \{(\alpha_i,m_i-\widehat{m}_i)\mid
m_i>\widehat{m}_i\}=\{(\beta_j,n_j-\widehat{n}_j)\mid
n_j>\widehat{n}_j\},
\]
so that $\alpha=\widehat{\alpha}\gamma$ and
$\beta=\widehat{\beta}\gamma$.

Any class $Z\in H_2(Y,\alpha,\beta)$ can be uniquely writen as
$Z=Z_0 + \widehat{Z}$, where $Z_0\in H_2(Y,\gamma,\gamma)$ corresponds
to $0$ under the obvious identification $H_2(Y,\gamma,\gamma)=H_2(Y)$,
and $\widehat{Z}\in H_2(Y,\widehat{\alpha},\widehat{\beta})$.  And
$\widehat{Z}$ has a nice representative by Lemma~\ref{lem:nice}.

\begin{lemma}
\label{lem:Qwetahat}
Given $Z\in H_2(Y,\alpha,\beta)$, let $\widehat{S}$ be a nice
representative of the corresponding class $\widehat{Z}\in
H_2(Y,\widehat{\alpha},\widehat{\beta})$, with associated braids
$\widehat{\zeta}_i^+$ and $\widehat{\zeta}_j^-$.  Choose the
trivializations so that $\tau_i^+=\tau_j^-$ whenever
$\alpha_i=\beta_j$.  Define
\[
\ell_\tau(\widehat{S},\R\times\gamma) \eqdef
\sum_i(m_i-\widehat{m_i})\eta_{\tau_i^+}(\widehat{\zeta}_i^+) -
\sum_j(n_j-\widehat{n_j})\eta_{\tau_j^-}(\widehat{\zeta}_j^-).
\]
Then
\[
Q_\tau(Z) = -w_\tau(\widehat{S}) - \eta_\tau(\widehat{S})
-2\ell_\tau(\widehat{S},\R\times\gamma).
\]
\end{lemma}

\begin{proof}
  It follows easily from the definition of $Q$, and our assumption
  that $\tau_i^+=\tau_j^-$ whenever $\alpha_i=\beta_j$, that
  $Q_\tau(Z_0)=0$.  So by equation \eqref{eqn:quadratic} and
  Lemma~\ref{lem:Qweta}, it is enough to show that
\begin{equation}
\label{eqn:QZZ0}
Q_\tau(\widehat{Z},Z_0) =
- \ell_\tau(\widehat{S},\R\times\gamma).
\end{equation}

We can find an admissible representative $S_0$ of the class $Z_0$
which is contained in a union of disjoint tubular neighborhoods of the
cylinders $\R\times\alpha_i$ for those Reeb orbits $\alpha_i$ that
equal some $\beta_j$.  Since the interior of $\widehat{S}$ does not
intersect the $\alpha_i$'s or $\beta_j$'s, we can arrange for the
interior of $S_0$ to be disjoint from the interior of $\widehat{S}$,
so that
\[
\#(\dot{\widehat{S}},\dot{S}_0)=0.
\]
At the same time we can arrange that the braids associated to $S_0$
are contained in tubular neighborhoods of the $\alpha_i$'s and
$\beta_j$'s that do not intersect the braids $\widehat{\zeta}_i^+$ and
$\widehat{\zeta}_j^-$, which implies that
\[
\ell_\tau(\widehat{S},S_0) =
\ell_\tau(\widehat{S},\R\times\gamma).
\]
Putting the above two equations into the definition of $Q$ proves
\eqref{eqn:QZZ0}.
\end{proof}

\subsection{The absolute grading determines the relative}
\label{sec:IRAP}

\begin{proof}[Proof of Theorem~\ref{thm:AG}(c)]
Pick an arbitrary class $Z\in H_2(Y,\alpha,\beta)$.  Choose
trivializations $\tau_+=\{\tau_i^+\}$ of $\xi$ over the $\alpha_i$'s
and trivializations $\tau_-=\{\tau_j^-\}$ of $\xi$ over the
$\beta_j$'s.  By the definitions of the relative and absolute versions
of $I$, to prove the desired identity \eqref{eqn:IRA}, we need to show
that
\begin{equation}
\label{eqn:IRAR}
P_{\tau_+}(L_+) - P_{\tau_-}(L_-) - \sum_iw_{\tau_i^+}(\zeta_i^+) +
\sum_jw_{\tau_j^-}(\zeta_j^-)
 \equiv c_\tau(Z) +
Q_\tau(Z)
\end{equation}
modulo $d(c_1(\xi) + 2\op{PD}([\Gamma]))$, where $\zeta_i^+$ is some
braid around $\alpha_i$ with $m_i$ strands, and $\zeta_j^-$ is some
braid around $\beta_j$ with $n_j$ strands, and $L_+\eqdef
\sqcup_i\zeta_i^+$ and $L_j\eqdef \sqcup_j\zeta_j^-$.

We will prove \eqref{eqn:IRAR} for special braids $\zeta_i^+$ and
$\zeta_j^-$ chosen as follows.  First define $\widehat{\alpha}$ and
$\widehat{\beta}$ by ``dividing $\alpha$ and $\beta$ by their greatest
common factor'' as explained in \S\ref{sec:Qweta}.  By
Lemma~\ref{lem:nice}, we can find a nice representative $\widehat{S}$
of the class $\widehat{Z}\in H_2(Y,\widehat{\alpha},\widehat{\beta})$
determined by $Z$.  Take the projection of $\widehat{S}$ to $Y$, and
remove its intersection with a union of small disjoint tubular
neighborhoods of the $\alpha_i$'s and $\beta_j$'s, to obtain an
embedded compact oriented surface $S$ in $Y$, whose boundary is a
transverse link.  More precisely,
\[
\partial S = \bigsqcup_i \widehat{\zeta}_i^+ \sqcup \bigsqcup_j -
\widehat{\zeta}_j^-,
\]
where $\widehat{\zeta}_i^+$ is a braid around $\alpha_i$ with
$\widehat{m}_i$ strands which does not intersect $\alpha_i$, and
$\widehat{\zeta}_j^-$ is a braid around $\beta_j$ with $\widehat{n}_j$
strands which does not intersect $\beta_j$.

Now define $\zeta_i^+$ and $\zeta_j^-$ as follows.  If $\alpha_i$ does
not equal any $\beta_j$, take $\zeta_i^+\eqdef \widehat{\zeta}_i^+$;
likewise if $\beta_j$ does not equal any $\alpha_i$, take $\zeta_j^-
\eqdef \widehat{\zeta}_j^-$.  If $\alpha_i=\beta_j$, choose an
arbitrary braid $\zeta_{ij}$ around this Reeb orbit with
$m_i-\widehat{m}_i=n_j-\widehat{n}_j$ strands, such that $\zeta_{ij}$ is
contained in a small tubular neighborhood of $\alpha_i=\beta_j$ which
does not intersect $\widehat{\zeta}_i^+$ or $\widehat{\zeta}_j^-$;
then take $\zeta_i^+ \eqdef \widehat{\zeta}_i^+\sqcup\zeta_{ij}$ and
$\zeta_j^-\eqdef \widehat{\zeta}_j^-\sqcup\zeta_{ij}$.

We now prove \eqref{eqn:IRAR} for these choices.  By equation
\eqref{eqn:writheUnion}, if $\alpha_i=\beta_j$, then
\[
\begin{split}
w_{\tau_i^+}(\zeta_i^+) &= w_{\tau_i^+}(\widehat{\zeta}_i^+) +
w_{\tau_i^+}(\zeta_{i,j}) + 2(m_i -
  \widehat{m}_i)\eta_{\tau_i^+}(\widehat{\zeta}_i^+),\\
w_{\tau_j^-}(\zeta_j^-) &= w_{\tau_j^-}(\widehat{\zeta}_j^-) +
w_{\tau_j^-}(\zeta_{i,j}) + 2(n_j -
  \widehat{n}_j)\eta_{\tau_j^-}(\widehat{\zeta}_j^-).
\end{split}
\]
Thus if we choose the trivializations so that $\tau_i^+=\tau_j^-$
whenever $\alpha_i=\beta_j$, then in the notation of Lemma~\ref{lem:Qwetahat},
\[
\sum_i w_{\tau_i^+}(\zeta_i^+) - \sum_j w_{\tau_j^-}(\zeta_j^-) =
w_\tau(\widehat{S}) + 2\ell_\tau(\widehat{S},\R\times\gamma).
\]
Thus by Lemma~\ref{lem:Qwetahat}, our goal \eqref{eqn:IRAR} is
equivalent to
\begin{equation}
\label{eqn:IRAE}
P_{\tau_+}(L_+) - P_{\tau_-}(L_-) \equiv c_\tau(Z) - \eta_\tau(\widehat{S}).
\end{equation}

To prove \eqref{eqn:IRAE}, let $\tau^\nu$ denote the framing of
$L_\pm$ induced by the conormal direction to $S$, together with some
fixed framings of the braids $\zeta_{i,j}$.  Then by
Lemma~\ref{lem:PtauL}(d),
\[
P_{\tau^\nu}(L_+) - P_{\tau^\nu}(L_-) \equiv
c_1(\xi|_{S},\tau^\nu).
\]
Now on each component $C$ of the braid $\widehat{\zeta}_i^+$, the
conormal framing $\tau^\nu$ differs from the framing
induced by $\tau_i^+$ by
the winding number $\eta_{\tau_i^+}(C)$.  Likewise for the braids
$\widehat{\zeta}_j^-$.  This framing difference has two consequences.  First,
\[
c_1(\xi|_{S},\tau^\nu)  = c_\tau(Z) + \eta_\tau(\widehat{S}).
\]
Second, using Lemma~\ref{lem:PtauL}(b),
\[
(P_{\tau^\nu}(L_+) - P_{\tau^\nu}(L_-)) - (P_{\tau_+}(L_+) -
P_{\tau_-}(L_-)) \equiv 2\eta_\tau(\widehat{S}).
\]
Combining the above three equations proves \eqref{eqn:IRAE}.
\end{proof}

\section{The index inequality in cobordisms}
\label{sec:IICob}

The main result of this section is Theorem~\ref{thm:indI} below, which
generalizes the basic ECH index inequality \eqref{eqn:II} to
symplectic cobordisms.

\subsection{Cobordism setup}

Let $(Y_+,\lambda_+,\omega_+)$ and $(Y_-,\lambda_-,\omega_-)$ be
closed oriented 3-manifolds with stable Hamiltonian structures.  Write
$E_+\eqdef [0,\infty)\times Y_+$ and $E_-\eqdef (-\infty,0]\times
Y_-$.  Let $s$ denote the $[0,\infty)$ or
$(-\infty,0]$ coordinate on $E_\pm$.

\begin{definition}
  A {\bf symplectic cobordism\/} from $Y_+$ to $Y_-$ is a smooth
  4-manifold $X$ with a decomposition
\[
X = E_- \cup_{Y_-} \overline{X} \cup_{Y_+}
E_+,
\]
where $(\overline{X},\omega)$ is a compact symplectic 4-manifold such
that $\partial \overline{X} = -Y_- \sqcup Y_+$ and $\omega|_{Y_\pm} =
\omega_\pm$.  In the special case
$(Y_\pm,\lambda_\pm,\omega_\pm)=(Y,\lambda,\omega)$, we allow
$\overline{X}=\emptyset$, in which case $X=\R\times Y$ is called the
{\bf symplectization\/} of $Y$.
\end{definition}

We will not really use the symplectic form on $\overline{X}$ in the
present paper, but it enables compactness results for holomorphic
curves \cite{behwz,cm}.

\begin{definition}
Let $X$ be a symplectic cobordism from $(Y_+,\lambda_+,\omega_+)$ to
$(Y_-,\lambda_-,\omega_-)$.  An almost complex structure $J$ on $X$ is
{\bf admissible\/} if:
\begin{itemize}
\item On $E_\pm$, the almost complex structure $J$ is independent of
  $s$, sends $\partial_s$ to $R_\pm$, and sends $\xi_\pm$ to itself
  compatibly with $\omega_\pm$.
\item
On $\overline{X}$, the almost complex structure $J$ is tamed by
$\omega$.
\end{itemize}
\end{definition}

\subsection{The ECH index in cobordisms}
\label{sec:ECHIC}

Fix a symplectic cobordism as above.  Suppose
$\alpha^+=\{(\alpha_i^+,m_i^+)\}$ is an orbit set in $Y_+$, and
$\alpha^-=\{(\alpha_j^-,m_j^-)\}$ is an orbit set in $Y_-$, such that
$[\alpha^+]\in H_1(Y_+)$ and $[\alpha^-]\in H_1(Y_-)$ map to the same
homology class in $H_1(X)$.  Let $H_2(X,\alpha^+,\alpha^-)$ denote the
set of relative homology classes of 2-chains $Z$ in the $4$-manifold
$X$ with
\[
\partial Z = \sum_i m_i^+ \{1\}\times\alpha_i^+ - \sum_j m_j^-
\{-1\}\times\alpha_j^-.
\]
This is an affine space over $H_2(X)$.  Note that in the special case when
 $X=\R\times Y$, this is canonically isomorphic to the affine
space $H_2(Y,\alpha^+,\alpha^-)$ from Definition~\ref{def:RHC}, via
the projection $\R\times Y\to Y$.

Returning to the general case, let $Z\in H_2(X,\alpha^+,\alpha^-)$.
If $\tau$ is a homotopy class of trivialization of $\xi_+$ over the
Reeb orbits $\alpha_i^+$ and of $\xi_-$ over the Reeb orbits
$\alpha_j^-$, define the relative first Chern class
\[
c_\tau(Z) \eqdef c_1(TX|_Z,\tau)\in\Z,
\]
generalizing Definition~\ref{def:ctau}, as follows.  Regard $TX$ as a
complex vector bundle via any admissible almost complex structure.
Fix a trivialization
\[
TX|_{\{1\}\times\alpha_i^+} \stackrel{\simeq}{\longrightarrow}
\alpha_i^+\times(\C\oplus\C)
\]
sending $\xi_+$ to the first summand via $\tau$ and sending
$\partial_s$ and $R_+$ to $1$ and $\sqrt{-1}$ respectively in the
second summand.  Choose an analogous trivialization of $TX$ over
$\{-1\}\times\alpha_j^-$.  Represent $Z$ by a smooth map $f:S\to X$
where $S$ is a compact surface with boundary.  Choose a generic
section $\psi$ of $f^*(\wedge^2 TX)$ which on $\partial S$ is
nonvanishing and has winding number zero with respect to the above
trivialization. Then define $c_1(TX|_Z,\tau)\eqdef \#\psi^{-1}(0)$.

If $Z\in H_2(X,\alpha^+,\alpha^-)$ and $Z'\in
H_2(X,{\alpha^+}',{\alpha^-}')$, and if $\tau$ is a trivialization of
$\xi_\pm$ over all orbits in $\alpha^\pm$ and ${\alpha^\pm}'$, then
$
Q_\tau(Z,Z')\in\Z
$
is defined by obvious analogy with Definition~\ref{def:Q}.

\begin{definition}
If $Z\in H_2(X,\alpha^+,\alpha^-)$, define the ECH index
\[
I(Z) \eqdef c_\tau(Z) + Q_\tau(Z) + \mu_\tau(\alpha^+) -
\mu_\tau(\alpha^-).
\]
\end{definition}

\subsection{Holomorphic curves}

Fix a symplectic cobordism $X$ with an admissible almost complex
structure $J$.  Recall that a {\bf holomorphic curve\/} in $X$ is a map
\[
u:(C,j) \longrightarrow (X,J)
\]
where $(C,j)$ is a Riemann surface and $J\circ du = du \circ j$.
One declares that $u:(C,j)\to(X,J)$ is equivalent to
$u':(C',j')\to(X,J)$ if and only if there is a biholomorphic map
$\varphi:(C,j)\to(C',j')$ such that $u'\circ\varphi=u$.

In this paper we will always assume further that:
\begin{itemize}
\item
$(C,j)$ is a punctured compact Riemann surface, possibly
disconnected.
\item
$u$ is nonconstant on each component of $C$.
\item Each end of $u$ is either asymptotic to $[0,\infty)\times\gamma$
  for some Reeb orbit $\gamma$ in $Y_+$, or asymptotic to
  $(-\infty,0]\times\gamma$ for some Reeb orbit $\gamma$ in $Y_-$.
\end{itemize}

If $\gamma$ is an embedded Reeb orbit in $Y_+$ and $k$ is a positive
integer, a {\bf positive end\/} of $u$ at $\gamma$ of multiplicity $k$
is an end of $u$ which is asymptotic to $[0,\infty)\times\gamma^k$.
Recall here that $\gamma^k$ denotes the $k$-fold iterate of $\gamma$.
Likewise, if $\gamma$ is an embedded Reeb orbit in $Y_-$, a {\bf
  negative end\/} of $u$ at $\gamma$ of multiplicity $k$ is an end of
$u$ which is asymptotic to $(-\infty,0]\times\gamma^k$.

\begin{definition}
  A holomorphic curve $u:C\to X$ is {\bf multiply covered\/} if there
  is a subset $C'\subset C$ which is a union of components, a
  holomorphic branched cover $\varphi:C'\to C_0$ of degree $>1$, and a
  holomorphic map $u_0:C_0\to X$ such that $u|_{C'}=u_0\circ\varphi$.
  Otherwise $u$ is called {\bf simple\/}.  We say that $u$ is {\bf
    irreducible\/} if its domain $C$ is connected.
\end{definition}

Let $\alpha^+=\{(\alpha_i^+,m_i^+)\}$ and
$\alpha^-=\{(\alpha_j^-,m_j^-)\}$ be orbit sets in $Y_+$ and $Y_-$
with the same homology class in $X$.

\begin{definition}
Let $\mc{M}(\alpha^+,\alpha^-)$ denote the moduli space of holomorphic
curves $u$ in $X$ with:
\begin{itemize}
\item
positive ends at $\alpha_i^+$ with total
multiplicity $m_i^+$, for each $i$;
\item
negative ends at $\alpha_j^-$ with total
multiplicity $m_j^-$, for each $j$;
\end{itemize}
and no other ends.
\end{definition}

\noindent
Any such $u$ determines a relative homology class $[u]\in
H_2(X,\alpha^+,\alpha^-)$, after using orientation-preserving
diffeomorphisms $[0,\infty)\simeq [0,1)$ and $(-\infty,0]\simeq
(-1,0]$ to identify
\[
X \simeq ((-1,0]\times Y_-) \cup_{Y_-} \overline{X} \cup_{Y_+} (([0,1)\times
Y_+).
\]

\begin{definition} Given $Z\in H_2(X,\alpha^+,\alpha^-)$, let
\[
\mc{M}(\alpha^+,\alpha^-,Z) \eqdef \{u\in \mc{M}(\alpha^+,\alpha^-)
\mid [u]=Z\}.
\]
\end{definition}

\begin{notation}
We will often abuse notation and refer to the holomorphic curve
$u:(C,j)\to(X,J)$ simply by $C$.  If $\tau$ is a trivialization of
$\xi$ over the Reeb orbits $\alpha^+_i$ and $\alpha^-_j$, we write
$c_\tau(C) \eqdef c_\tau([C])$; $Q_\tau(C)\eqdef Q_\tau([C])$;
$\mu_\tau(C) \eqdef \mu_\tau(\alpha^+)-\mu_\tau(\alpha^-)$; and
\[
I(C)\eqdef I(\alpha^+,\alpha^-,[C]) = c_\tau(C)+Q_\tau(C)+\mu_\tau(C).
\]
\end{notation}

\begin{example}
If $C$ is closed, representing a homology class $[C]\in H_2(X)$, then
\[
I(C) = \langle c_1(TX),[C]\rangle + [C]\cdot [C].
\]
Taubes's Gromov invariant \cite{gr} of a closed symplectic 4-manifold $X$
counts (in a subtle way) holomorphic curves $C$ in $X$ with $I(C)=0$.
\end{example}

\subsection{The relative adjunction formula}

Consider now a {\em simple\/} $J$-holomorphic curve
$C\in\mc{M}(\alpha^+,\alpha^-,Z)$.  It follows from
\cite[Cor. 2.6]{siefring1} that $C$ is embedded except possibly for
finitely many singularities.  We then have the following relative
adjunction formula:

\begin{proposition}
If $C$ is a simple holomorphic curve in $X$ as above, then
\begin{equation}
\label{eqn:RAF}
c_\tau(C) = \chi(C) + Q_\tau(C) + w_\tau(C) - 2 \delta(C).
\end{equation}
\end{proposition}

\noindent
Here $\delta(C)$ is a count of the singularities of $C$ in $X$ with
positive integer weights as in \cite[\S7]{mw}.  The weight of a singular
point $p$ is the number of self-intersections of a perturbation of $C$
to a generic holomorphic immersion in a neighborhood of $p$.  Also,
$w_\tau(C)$ is the {\bf asymptotic writhe\/} of $C$, defined by
obvious analogy with \eqref{eqn:wtauS}.  A proof of the relative
adjunction formula in a slightly different context can be found in
\cite[\S3]{pfh2}, and this carries over in a straightforward manner to
the present situation.

\begin{example}
  If $C$ is closed, then there is no writhe term or trivialization
  choice, and \eqref{eqn:RAF} reduces to the usual adjunction formula
\begin{equation}
\label{eqn:closedAdjunction}
\langle c_1(TX),[C]\rangle = \chi(C) + [C]\cdot[C] - 2 \delta(C).
\end{equation}
\end{example}

\subsection{The Fredholm index}
\label{sec:FI}

Let $C\in\mc{M}(\alpha^+,\alpha^-)$.  For each $i$, let $n_i^+$ denote
the number of positive ends of $u$ at $\alpha_i^+$, and let
$\{q^+_{i,k}\}_{k=1}^{n_i^+}$ denote their multiplicities.  Likewise,
for each $j$, let $n_j^-$ denote the number of negative ends of $u$ at
$\alpha_j^-$, and let $\{q^-_{j,k}\}_{k=1}^{n_j^-}$ denote their
multiplicities. Thus $\sum_{k=1}^{n_i^+}q^+_{i,k}=m_i^+$ and
$\sum_{k=1}^{n_j^-}q^-_{j,k}=m_j^-$.

\begin{notation}If $\tau$ is a trivialization of $\xi_\pm$ over the
  orbits in $\alpha^\pm$, define
\[
\mu_\tau^0(C) \eqdef
\sum_i\sum_{k=1}^{n_i^+}\op{CZ}_\tau((\alpha_i^+)^{q_{i,k}^+}) -
\sum_j\sum_{k=1}^{n_j^-}\op{CZ}_\tau((\alpha_j^-)^{q_{j,k}^-}).
\]
That is, $\mu_\tau^0(C)$ is the sum of the Conley-Zehnder indices of
the positive ends of $C$, minus the sum of the CZ indices of the
negative ends of $C$.  This should be contrasted with $\mu_\tau(C)$,
which is a sum of many more Conley-Zehnder terms:
\[
\mu_\tau(C) = \sum_i\sum_{l=1}^{m_i^+}\op{CZ}_\tau((\alpha_i^+)^l) -
\sum_j\sum_{l=1}^{m_j^-}\op{CZ}_\tau((\alpha_j^-)^l).
\]
\end{notation}

\begin{definition}
Define the {\bf Fredholm index\/}
\begin{equation}
\label{eqn:ind}
\op{ind}(C) \eqdef -\chi(C) + 2 c_\tau(C) + \mu_\tau^0(C).
\end{equation}
\end{definition}

It is shown in \cite{dragnev}, using an index formula from
\cite{schwarz}, that if $J$ is generic and $C$ is simple, then the
moduli space $\mc{M}(\alpha^+,\alpha^-,[C])$ is a manifold near $C$ of
dimension $\op{ind}(C)$.  However in the present paper we do not
assume that $J$ is generic.

\subsection{Incoming and outgoing partitions}

Before stating the index inequality, we need a digression to introduce
some special partitions associated to Reeb orbits.

Let $Y$ be a three-manifold with a stable Hamiltonian structure, let $\gamma$
be an embedded Reeb orbit in $Y$, and let $m$ be a positive integer.

\begin{definition}
\label{def:partitions}
  Define two partitions of $m$, the {\bf incoming partition\/}
  $\pin_\gamma(m)$ and the {\bf outgoing partition\/}
  $\pout_\gamma(m)$, as follows.
\begin{itemize}
\item
If $\gamma$ is positive hyperbolic, then
\[
\pin_\gamma(m)\eqdef \pout_\gamma(m) \eqdef (1,\ldots,1).
\]
\item
If $\gamma$ is negative hyperbolic, then
\[
\pin_\gamma(m) \eqdef \pout_\gamma(m) \eqdef \left\{\begin{array}{cl}
    (2,\ldots,2), & \mbox{$m$ even,}\\ (2,\ldots,2,1), & \mbox{$m$
      odd.} \end{array}
\right.
\]
\item
If $\gamma$ is elliptic with monodromy angle $\theta$, then
$\pin_\gamma(m) \eqdef \pin_\theta(m)$ and $\pout_\gamma(m)\eqdef
\pout_\theta(m)$, where $\pin_\theta(m)$ and $\pout_\theta(m)$ are
defined below.
\end{itemize}
\end{definition}

\begin{definition}
\label{def:partEll}
Let $\theta$ be an irrational number and let $m$ be a positive
integer.  Define partitions $\pin_\theta(m)$ and $\pout_\theta(m)$ of
$m$ as follows.

Let $\lin_\theta(m)$ denote the lowest convex polygonal path in the
plane that starts at $(0,0)$, ends at $(m,\ceil{m\theta})$, stays
  above the line $y=\theta x$, and has corners at lattice points.
  Then the integers in $\pin_\theta(m)$ are the horizontal
  displacements of the segments of the path $\lin_\theta(m)$ between
  lattice points.

Likewise, let $\lout_\theta(m)$ denote the highest concave polygonal
path in the plane that starts at $(0,0)$, ends at
$(m,\floor{m\theta})$, stays below the line $y=\theta x$, and has
corners at lattice points.  Then the integers in $\pout_\theta(m)$ are
the horizontal displacements of the segments of the path
$\lout_\theta(m)$ between lattice points.
\end{definition}

Note that $\pin_\theta(m)$ and $\pout_\theta(m)$ depend only on the
class of $\theta$ in $\R/\Z$.  Also, $\pin_\theta(m) =
\pout_{-\theta}(m)$.  For more about the incoming and outgoing
partitions, see \cite[\S4]{pfh2} and \cite[\S7]{obg1}.

\subsection{Statement of the index inequality}
\label{sec:indI}

Fix a symplectic cobordism $X$ with an admissible almost complex
structure $J$ (not necessarily generic).  Continue with the notation
from \S\ref{sec:FI}.

\begin{theorem}
\label{thm:indI}
Suppose $C\in\mc{M}(\alpha^+,\alpha^-)$ is simple.  Then
\begin{equation}
\label{eqn:indI}
\op{ind}(C) \le I(C) - 2\delta(C).
\end{equation}
Equality holds only if $\{q^+_{i,k}\}=\pout_{\alpha_i^+}(m_i^+)$ for
each $i$, and $\{q^-_{j,k}\}=\pin_{\alpha_j^-}(m_j^-)$ for each $j$.
\end{theorem}

The proof of Theorem~\ref{thm:indI} has three ingredients.  The first
ingredient is the relative adjunction formula \eqref{eqn:RAF}, which
implies that the index inequality \eqref{eqn:indI} is equivalent to
the writhe bound
\begin{equation}
\label{eqn:writheBound}
w_\tau(C) \le \mu_\tau(C) - \mu_\tau^0(C).
\end{equation}
The second ingredient is an analytic bound on the writhe $w_\tau(C)$,
and the third ingredient is a combinatorial inequality.  We now
explain these.

\subsection{The analytic writhe bound}

Fix an embedded Reeb orbit $\gamma$ in $Y_+$ at which $C$ has positive
ends of multiplicities $q_1,\ldots,q_n$ with total multiplicity $m$.
These ends determine a braid $\zeta$ around $\gamma$ with components
$\zeta_1,\ldots,\zeta_n$, where $\zeta_i$ has $q_i$ strands.  The
braid $\zeta$ is the intersection of $C$ with $\{R\}\times N$, where
$R>>0$ and $N\subset Y_+$ is a small tubular neighborhood of $\gamma$.

Now fix a trivialization $\tau$ of $\xi$ over $\gamma$.  We have the
following two key analytic lemmas about the writhes and linking
numbers of the braids $\zeta_i$.  To simplify notation, write
\[
\rho_i \eqdef \floor{\frac{\op{CZ}_\tau(\gamma^{q_i})}{2}}.
\]

\begin{lemma}
\label{lem:WBE}
Let $i\in\{1,\ldots,n\}$.  Then
\begin{equation}
\label{eqn:WBE}
w_\tau(\zeta_i) \le
\rho_i(q_i-1).
\end{equation}
Equality holds only if:
\begin{description}
\item{(i)}
If $\gamma$ is positive hyperbolic, then $q_i=1$.
\item{(ii)}
If $\gamma$ is negative hyperbolic, then $q_i$ is odd or $q_i=2$.
\end{description}
\end{lemma}

(The inequality \eqref{eqn:WBE} can sometimes be improved when
$\rho_i$ and $q_i$ have a common factor.  The necessary conditions
(i) and (ii) for equality are a special case of this improvement.)

\begin{lemma}
\label{lem:linkingBound}
Let $i,j\in\{1,\ldots,n\}$ be distinct.  Then
\[
\ell_\tau(\zeta_i,\zeta_j) \le
\max\left(q_i\rho_j,
q_j\rho_i\right).
\]
\end{lemma}

These lemmas were proved in \cite[\S6]{pfh2} in an easier
setting\footnote{Note that \cite[\S6]{pfh2} discusses negative ends
instead of positive ends, but this is completely analogous.} \footnote{The
proof of Lemma~\ref{lem:linkingBound} in \cite{pfh2} assumed that
neither of the positive ends of $C$ corresponding to $\zeta_i$ or
$\zeta_j$ is ``trivial'', i.e. of the form $[0,\infty)\times\gamma$.
But this assumption is easily dropped: if the $i$ end is trivial, then
$q_i=1$ and $\ell_\tau(\zeta_i,\zeta_j) = \eta_\tau(\zeta_j)$, and a
fundamental winding number bound from \cite{hwz2} asserts that
$\eta_\tau(\zeta_j)\le\rho_j$.}.  The asymptotic analysis necessary to
carry them over to the present setting was done by Siefring
\cite[Theorems 2.2 and 2.3]{siefring1}.  Combining the above two
lemmas and using equation \eqref{eqn:writheUnion}, we obtain
the following bound on the writhe of $\zeta$:

\begin{lemma}
\label{lem:writheBound}
\begin{equation}
\label{eqn:mainWritheBound}
w_\tau(\zeta) \le \sum_{i,j=1}^n \max(q_i\rho_j,q_j\rho_i) -
\sum_{i=1}^n\rho_i.
\end{equation}
Equality holds only if conditions (i) and (ii) in Lemma~\ref{lem:WBE}
hold.
\end{lemma}

To understand the right hand side of \eqref{eqn:mainWritheBound}, we will
shortly prove the following combinatorial lemma:

\begin{lemma}
\label{lem:CE0}
\begin{equation}
\label{eqn:CI}
\sum_{i,j=1}^n \max(q_i\rho_j,q_j\rho_i) -
\sum_{i=1}^n\rho_i
\le
\sum_{k=1}^m \op{CZ}_\tau(\gamma^k) -
\sum_{i=1}^n\op{CZ}_\tau(\gamma^{q_i}).
\end{equation}
Equality holds if and only if:
\begin{description}
\item{(iii)}
If $\gamma$ is negative hyperbolic, then all $q_i$'s are even, except
that one $q_i$ might equal $1$.
\item{(iv)}
If $\gamma$ is elliptic with monodromy angle $\theta$, then
$(q_1,\ldots,q_n) = \pout_\theta(m)$.
\end{description}
\end{lemma}

Granted this lemma, combining it with Lemma~\ref{lem:writheBound}
gives

\begin{lemma}
\label{lem:UWB}
\[
w_\tau(\zeta) \le \sum_{k=1}^m \op{CZ}_\tau(\gamma^k) -
\sum_{i=1}^n\op{CZ}_\tau(\gamma^{q_i}),
\]
with equality only if $(q_1,\ldots,q_k)=\pout_\gamma(m)$.
\end{lemma}

This lemma implies Theorem~\ref{thm:indI}, because combining these
inequalities for all the orbits in $\alpha^+$, along with analogous
inequalities for the orbits in $\alpha^-$, shows that the inequality
\eqref{eqn:writheBound} holds, with equality only under the conditions
stipulated in Theorem~\ref{thm:indI}.

\subsection{Proof of the combinatorial lemma \ref{lem:CE0}}

The proof of Lemma~\ref{lem:CE0} is easy when $\gamma$ is positive
hyperbolic, because by \eqref{eqn:CZHyp} we can choose the
trivialization $\tau$ so that $\op{CZ}_\tau(\gamma^k)=0$ for all $k$,
and then both sides of the inequality \eqref{eqn:CI} vanish.

If $\gamma$ is negative hyperbolic, then we can choose the trivialization
$\tau$ so that $\op{CZ}_\tau(\gamma^k)=k$.  It is convenient to order
the $q_i$'s so that $q_1,\ldots,q_k$ are odd, $q_{k+1},\ldots,q_n$
are even, and $q_1\ge \cdots \ge q_k$.  Then a straightforward
computation shows that the inequality \eqref{eqn:CI} is equivalent to
\[
\sum_{i=1}^k\left(1 - i +
\frac{1-q_i}{2}\right) \le 0.
\]
Lemma~\ref{lem:CE0} follows immediately in this case.

Finally, if $\gamma$ is elliptic with monodromy angle $\theta$, then by
equation \eqref{eqn:CZEll}, Lemma~\ref{lem:CE0} reduces to the
following lemma.  This was proved in \cite{pfh2}, but we will give a new
and more transparent proof here.

\begin{lemma}
\label{lem:CE1}
Let $\theta$ be an irrational number, and let $q_1,\ldots,q_n$ be
positive integers with $m\eqdef \sum_{i=1}^n q_i$.  Then
\[
\sum_{i,j=1}^n\max(q_i\floor{q_j\theta},q_j\floor{q_i\theta}) \le
2\sum_{k=1}^m\floor{k\theta} - \sum_{i=1}^n\floor{q_i\theta} + m - n.
\]
Equality holds if and only if $(q_1,\ldots,q_k) = \pout_\theta(m)$.
\end{lemma}

\begin{proof}
Order the $q_i$'s so that
\begin{equation}
\label{eqn:ordering}
\frac{\floor{q_1\theta}}{q_1} \ge \frac{\floor{q_2\theta}}{q_2} \ge
\cdots \ge \frac{\floor{q_n\theta}}{q_n}.
\end{equation}
Let $\Lambda$ denote the rightward-pointing polygonal path in the
plane connecting the lattice points
\[
\sum_{i=1}^{j} \left(q_i,\floor{q_i\theta}\right), \quad\quad
j=0,1,\ldots,n
\]
by line segments.  Let $R$ denote the region in the plane enclosed by
the path $\Lambda$, the horizontal line from the origin to
$(\sum_{i=1}^n q_i,0)$, and the vertical line from $(\sum_{i=1}^n
q_i,0)$ to $\sum_{i=1}^n(q_i, \floor{q_i\theta})$.  Let
$A$ denote the area of the region $R$, let $L$ denote the number of
lattice points in $R$ (including the boundary), and let $B$ denote
the number of lattice points on the boundary of $R$.

By our ordering convention \eqref{eqn:ordering}, the left side of the
inequality we want to prove is given by
\begin{equation}
\label{eqn:pick1}
\begin{split}
\sum_{i,j=1}^n\max(q_i\floor{q_j\theta},q_j\floor{q_i\theta}) & = 
\sum_{i=1}^n\floor{q_i\theta}\left(q_i +
  2\sum_{j=i+1}^nq_j\right) \\
&=2A,
\end{split}
\end{equation}
where the area of $R$ is computed by cutting it into rectangles and
triangles of height $\floor{q_i\theta}$ and base $q_j$.  On the other
hand, Pick's formula for the area of a lattice polygon tells us that
\begin{equation}
\label{eqn:pick2}
2A = 2L - B - 2.
\end{equation}
By dividing up the lattice points in $R$ into vertical lines, we find
that
\begin{equation}
\label{eqn:pick3}
L \le 1+ \sum_{k=1}^m(\floor{k\theta}+1),
\end{equation}
with equality if and only if the polygonal paths $\Lambda$ and
$\lout_\theta(m)$ have the same image.  Finally, the
number of boundary lattice points satisfies
\begin{equation}
\label{eqn:pick4}
B \ge m + n + \sum_{i=1}^n\floor{q_i\theta},
\end{equation}
with equality if and only if none of the edge vectors
$(q_i,\floor{q_i\theta})$ of the path $\Lambda$ is divisible in $\Z^2$.
The lemma follows immediately by combining
\eqref{eqn:pick1}--\eqref{eqn:pick4}.
\end{proof}  

\section{ECH index of unions and multiple covers}
\label{sec:IU}

\subsection{Statement of the result}
\label{sec:IN}

As in \S\ref{sec:IICob}, fix a symplectic cobordism $X$ with an
admissible almost complex structure $J$.  The main result of this
section is the following inequality regarding the ECH index of the
union of two holomorphic curves.

\begin{theorem}
\label{thm:IU}
If $C$ and $C'$ are holomorphic curves in $X$, then
\[
I(C\cup C') \ge I(C) + I(C') + 2C\cdot C'.
\]
Here $C\cdot C'$ is an ``intersection number'' of $C$ and $C'$
defined below.
\end{theorem}

\noindent
Note that $C$ and $C'$ are not assumed to be simple, irreducible, or
distinct.

\begin{definition}
  If $C$ and $C'$ are simple curves in $X$ with no irreducible
  component in common, define $C\cdot C'\in\Z$ to be the algebraic
  count of intersections of $C$ and $C'$.  This is well-defined,
  because it follows from \cite[Cor. 2.5]{siefring1} that there are
  only finitely many intersections.  Also each intersection counts
  positively \cite{mcds}.
\end{definition}

Now for the slightly nonstandard part of the definition.

\begin{definition}
If $C$ is a simple, irreducible holomorphic curve in $X$, define
\[
C\cdot C \eqdef \frac{1}{2}\big(2g(C) - 2 + \op{ind}(C) + h(C) +
  4\delta(C)\big) \in \frac{1}{2}\Z.
\]
Here $g(C)$ denotes the genus of $C$, and $h(C)$ denotes the number of
ends of $C$ at hyperbolic Reeb orbits.
\end{definition}

\begin{example}
  If $C$ is closed, then it follows from \eqref{eqn:closedAdjunction}
  and \eqref{eqn:ind} that $C\cdot C$ equals the usual homological
  intersection number $[C]\cdot [C]$.
\end{example}

\begin{remark}
  If $J$ is generic so that $\op{ind}(C)\ge 0$, then clearly $C\cdot C
  \ge -1$.  In a symplectization the situation is better (without any
  genericity assumption on $J$):

\begin{proposition}
\label{prop:CdotC}
If $X$ is a symplectization, if $C$ is a simple, irreducible
holomorphic curve in $X$, and if $C$ is not a cylinder
$\R\times\gamma$ where $\gamma$ is an elliptic Reeb orbit, then
$C\cdot C \ge 0$.
\end{proposition}

\begin{proof}
  If $C$ is a cylinder $\R\times\gamma$ where $\gamma$ is a hyperbolic
  Reeb orbit then it follows immediately from the definition that
  $C\cdot C=0$.  If $C$ is not a cylinder $\R\times\gamma$, then a
  stronger inequality than $C\cdot C\ge 0$ is known.  Namely, if
  $h_+(C)$ denotes the number of ends of $C$ at {\em positive\/}
  hyperbolic orbits only, then
\[
2g(C) - 2 + \op{ind}(C) + h_+(C) \ge 0,
\]
cf.\ \cite{hwz2} and \cite[Prop. 4.1]{wendl1}. 
The proof of this inequality uses a linearized version of positivity of
intersections of $C$ with its translates in the $\R$ direction.
\end{proof}
\end{remark}

\begin{definition}
If $C$ is a union of $d_a$-fold covers of distinct simple irreducible
curves $C_a$, and if $C'$ is a union of $d_b'$-fold covers of distinct
simple irreducible curves $C_b'$, then
\[
C \cdot C' \eqdef \sum_a\sum_bd_ad_b'C_a\cdot C_b'.
\]
\end{definition}

This completes the statement of Theorem~\ref{thm:IU}.

\begin{example}
  The inequality \eqref{eqn:trivCyl} is a special case of
  Theorem~\ref{thm:IU} in which $X$ is a symplectization and $C'$ is a
  union of $\R$-invariant cylinders and the image of $C$ contains no
  $\R$-invariant cylinder.  This was proved in \cite[Prop. 7.1]{pfh2}.
  In this case one also obtains necessary conditions for equality in
  terms of the incoming and outgoing partitions; for more about these
  conditions see \cite[Lem. 7.28]{obg1}.
\end{example}

\begin{remark}
  The proof of Theorem~\ref{thm:IU} will show that in some cases one
  can obtain a stronger inequality, in which $C\cdot C'$ is replaced
  by a more elaborate notion of intersection number involving
  additional contributions from the ends.  Some related notions of
  intersection number are discussed
  in \cite[\S4]{wendl2}, using work of Siefring \cite{su}.
\end{remark}

\subsection{Proof of Theorem~\ref{thm:IU}}

The proof of Theorem~\ref{thm:IU} will use the following combinatorial lemma.

\begin{lemma}
\label{lem:CLI}
Let $\gamma$ be a Reeb orbit, let $q_1,\ldots,q_n$ be
positive integers with $m\eqdef \sum_{i=1}^n q_i$, let
$q_1',\ldots,q'_{n'}$ be positive integers with $m'\eqdef
\sum_{j=1}^{n'}q_j'$, and let
$\rho_i\eqdef\floor{\op{CZ}_\tau(\gamma^{q_i})/2}$ and
$\rho_j'\eqdef\floor{\op{CZ}_\tau(\gamma^{q'_j})/2}$. Then
\begin{equation}
\label{eqn:clii}
2\sum_{i=1}^n\sum_{j=1}^{n'}\max(q_i\rho'_j,q_j'\rho_i) \le
\left(\sum_{k=1}^{m+m'} - \sum_{k=1}^m -
  \sum_{k=1}^{m'}\right)\op{CZ}_\tau(\gamma^k).
\end{equation}
\end{lemma}

\begin{proof}
We consider two cases.

Case 1: If $\gamma$ is hyperbolic, then by equation \eqref{eqn:CZHyp},
there is an integer $l$ such that $\op{CZ}_\tau(\gamma^k)=kl$ for all
$k$.  Then the inequality \eqref{eqn:clii} is equivalent to
\[
2\sum_{i=1}^n \sum_{j=1}^{n'} \max\left(q_i\floor{\frac{lq_j'}{2}},
  q_j'\floor{\frac{lq_i}{2}}\right) \le lmm'.
\]
But this inequality is obvious because
\begin{equation}
\label{eqn:obviousInequality}
 \max\left(q_i\floor{\frac{lq_j'}{2}},
  q_j'\floor{\frac{lq_i}{2}}\right)
\le \frac{lq_iq_j'}{2}.
\end{equation}

Case 2: If $\gamma$ is elliptic with mondromy angle $\theta$, then by
equation \eqref{eqn:CZEll}, the inequality \eqref{eqn:clii} is
equivalent to
\begin{equation}
\label{eqn:cliii}
\sum_{i=1}^n \sum_{j=1}^{n'} \max(q_i\floor{q_j'\theta},
q_j'\floor{q_i\theta}) \le \left(\sum_{k=1}^{m+m'} - \sum_{k=1}^m -
  \sum_{k=1}^{m'}\right)\floor{k\theta}.
\end{equation}
We prove this by induction on $n+n'$.  Without loss of generality,
$n\ge 1$, and $q_n^{-1}\floor{q_n\theta}$ is the smallest number in
the set $\{q_i^{-1}\floor{q_i\theta}\} \cup
\{{q_j'}^{-1}\floor{q_j'\theta}\}$.  Then
\begin{equation}
\label{eqn:inductiveStep}
\begin{split}
\sum_{j=1}^{n'}\max(q_n\floor{q_j'\theta},q_j'\floor{q_n\theta}) &=
q_n\sum_{j=1}^{n'}\floor{q_j'\theta}\\
&\le q_n\floor{m'\theta}\\
&\le \sum_{k=m-q_n+1}^{m}(\floor{(k+m')\theta} - \floor{k\theta}).
\end{split}
\end{equation}
This reduces \eqref{eqn:cliii} to the corresponding inequality with
$n$ decreased by $1$ and $q_n$ removed, so we are done by induction.
\end{proof}

\begin{proof}[Proof of Theorem~\ref{thm:IU}.]

Let $C_1,\ldots,C_r$ denote the distinct irreducible simple curves
that appear in the image of either $C$ or $C'$.  Thus $C$ consists of
a $d_a$-fold cover of $C_a$ for each $a=1,\ldots,r$, and $C'$ consists
of a $d_a'$-fold cover of $C_a$ for each $a$, for some $d_a,d_a'\ge
0$.

To prove the theorem, we begin by reducing to a local statement around
each Reeb orbit.  Since $c_\tau$ is additive, and $Q_\tau$ is
quadratic as in \eqref{eqn:quadratic}, we have
\[
I(C) = \sum_{a=1}^r d_ac_\tau(C_a) +
\sum_{a,b=1}^rd_ad_bQ_\tau(C_a,C_b) + \mu_\tau(C).
\]
Adding the analogous equation for $I(C')$, and subtracting the result
from the analogous equation for $I(C\cup C')$, we obtain
\[
I(C\cup C') - I(C) - I(C') = 2\sum_{a,b=1}^rd_ad_b'Q_\tau(C_a,C_b) +
\mu_\tau(C\cup C') - \mu_\tau(C) - \mu_\tau(C').
\]
By the definition of $Q$, if $a\neq b$ then
\[
Q_\tau(C_a,C_b) = C_a\cdot C_b - \ell_\tau(C_a,C_b).
\]
On the other hand, by the relative adjunction formula \eqref{eqn:RAF}
and the index formula \eqref{eqn:ind}, we have
\[
Q_\tau(C_a) = C_a\cdot C_a + \frac{1}{2}\left(e(C_a) -
  \mu_\tau^0(C_a) - 2w_\tau(C_a)\right),
\]
where $e(C_a)$ denotes the number of ends of $C_a$ at elliptic Reeb
orbits.  Putting this all together gives
\begin{equation}
\label{eqn:PTAT}
\begin{split}
I(C\cup C') - I(C) - I(C') - 2 C\cdot C' &= \mu_\tau(C\cup C') -
\mu_\tau(C) - \mu_\tau(C')\\
&\;\,-2\sum_{a\neq b}d_ad_b'\ell_\tau(C_a,C_b) \\
&\;\,+\sum_{a=1}^rd_ad_a'(e(C_a) -\mu_\tau^0(C_a) - 2w_\tau(C_a)).
\end{split}
\end{equation}
We need to prove that the right side of this equation is nonnegative.

Let $\gamma$ be a Reeb orbit in $Y_+$ at which some of the curves
$C_a$ have positive ends.  For $a=1,\ldots,r$, let $n_a\ge 0$ denote
the number of positive ends of $C_a$ at $\gamma$, let
$q_{a,1},\ldots,q_{a,n_a}$ denote the multiplicities of these ends,
let $m_a\eqdef \sum_{i=1}^{n_a}q_{a,i}$, and let $\zeta_a$ denote the
braid around $\gamma$ corresponding to these ends. Also let $M\eqdef
\sum_{a=1}^r d_am_a$ and $M'\eqdef \sum_{a=1}^r d_a'm_a$. Define $\varepsilon$
to equal $1$ if $\gamma$ is elliptic and $0$ if $\gamma$ is
hyperbolic.  Then it is enough to show that
\begin{equation}
\label{eqn:huge}
\begin{split}
\left(\sum_{k=1}^{M+M'} - \sum_{k=1}^M -
  \sum_{k=1}^{M'}\right)\op{CZ}_\tau(\gamma^k) & \ge  2 \sum_{a\neq
  b}d_ad_b'\ell_\tau(\zeta_a,\zeta_b) \\
+ & \sum_{a=1}^rd_ad_a'\left(-\varepsilon n_a
  + \sum_{i=1}^{n_a}\op{CZ}_\tau(\gamma^{q_{a,i}}) +
  2w_\tau(\zeta_a)\right).
\end{split}
\end{equation}
(We also need to prove an analogous inequality for the negative ends,
but this is completely symmetric.)  To prove the inequality
\eqref{eqn:huge}, we consider two cases.

Case 1: Suppose $\gamma$ is elliptic with monodromy angle $\theta$.
As usual, write
$\rho_{a,i}\eqdef\floor{\op{CZ}_\tau(\gamma^{q_{a,i}})/2}$.  For
$a,b=1,\ldots,r$, introduce the notation
\[
f(a,b) \eqdef \sum_{i=1}^{n_a}\sum_{j=1}^{n_b}
\max(q_{a,i}\rho_{b,j}, q_{b,j}\rho_{a,i})
\]
By
Lemma~\ref{lem:linkingBound}, for $a\neq b$ we have
\[
\ell_\tau(\zeta_a,\zeta_b) \le f(a,b).
\]
Since $\op{CZ}_\tau(\gamma^k)$ is odd, Lemma~\ref{lem:writheBound}
implies that
\[
2w_\tau(\zeta_a) \le 2 f(a,a) -
\sum_{i=1}^{n_a}\op{CZ}_\tau(\gamma^{q_{a,i}}) + n_a.
\]
So to prove \eqref{eqn:huge} in this case, it is enough to show that
\[
2\sum_{a,b=1}^rd_ad_b'f(a,b) \le \left(\sum_{k=1}^{M+M'} - \sum_{k=1}^M -
  \sum_{k=1}^{M'}\right)\op{CZ}_\tau(\gamma^k).
\]
But this follows by applying Lemma~\ref{lem:CLI} to the list
consisting of the numbers $q_{a,i}$ repeated $d_a$ times, and the list
consisting of the numbers $q_{a,i}$ repeated $d_a'$ times.

Case 2: Suppose $\gamma$ is hyperbolic.  For $a\neq b$, by
Lemmas~\ref{lem:linkingBound} and \ref{lem:CLI} we have
\[
2\ell_\tau(\zeta_a,\zeta_b) \le \left(\sum_{k=1}^{m_a+m_b} -
  \sum_{k=1}^{m_a} - \sum_{k=1}^{m_b}\right) \op{CZ}_\tau(\gamma^k).
\]
And by Lemma~\ref{lem:UWB}, we have
\[
w_\tau(\zeta_a) \le \sum_{k=1}^{m_a}\op{CZ}_\tau(\gamma^k) -
\sum_{i=1}^{n_a}\op{CZ}_\tau(\gamma^{q_{a,i}}).
\]
So to prove the inequality \eqref{eqn:huge} in this case, it is enough
to show that
\[
\begin{split}
\left(\sum_{k=1}^{M+M'} - \sum_{k=1}^M -
  \sum_{k=1}^{M'}\right)\op{CZ}_\tau(\gamma^k) & \ge
\sum_{a\neq
  b}d_ad_b' \left(\sum_{k=1}^{m_a+m_b} -
  \sum_{k=1}^{m_a} - \sum_{k=1}^{m_b}\right) \op{CZ}_\tau(\gamma^k)
\\
+ & \sum_{a=1}^rd_ad_a'\left(
  - \sum_{i=1}^{n_a}\op{CZ}_\tau(\gamma^{q_{a,i}}) +
  2\sum_{k=1}^{m_a}\op{CZ}_\tau(\gamma^k) \right).
\end{split}
\]
But a straightforward computation using equation \eqref{eqn:CZHyp}
shows that this last inequality is always an equality.
\end{proof}

\section{The ECH index and the Euler characteristic}
\label{sec:RF}

To put the previous results in perspective, we now introduce a natural
variant of the ECH index, which bounds the negative Euler
characteristic of holomorphic curves, and which gives rise to a
relative filtration on ECH, or any other kind of contact homology of a
contact 3-manifold.

\subsection{Definition of $J_0$, $J_+$, and $J_-$}

\begin{notation}
  In a 3-manifold with a stable Hamiltonian structure (in which all Reeb
  orbits are assumed nondegenerate as usual), if
  $\alpha=\{(\alpha_i,m_i)\}$ is an orbit set and if $\tau$ is a
  trivialization of $\xi$ over the $\alpha_i$'s, define
\[
\mu_\tau'(\alpha) \eqdef \sum_i
\sum_{k=1}^{m_i-1}\op{CZ}_\tau(\alpha_i^k).
\]
This differs from the quantity $\mu_\tau(\alpha)$ defined in
\S\ref{sec:DECHI} only in that $m_i$ there is replaced by $m_i-1$ here.
\end{notation}

Now let $X$ be a symplectic cobordism from $Y_+$
to $Y_-$, and continue with the notation from \S\ref{sec:ECHIC}.

\begin{definition}
  Let $\alpha^+$ be an orbit set in $Y_+$ and let $\alpha^-$ be an
  orbit set in $Y_-$ such that $[\alpha^+]=[\alpha^-]\in H_1(X)$, and
  let $Z\in H_2(X,\alpha^+,\alpha^-)$.  Define
\[
  J_0(\alpha^+,\alpha^-,Z) \eqdef -c_\tau(Z) + Q_\tau(Z) +
  \mu_\tau'(\alpha^+)
  - \mu_\tau'(\alpha^-).
\]
\end{definition}

Here $\tau$ is a trivialization of $\xi_+$ over the orbits in
$\alpha^+$ and of $\xi_-$ over the orbits in $\alpha^-$.  The
definition of $J_0$ differs from that of the ECH index $I$ only in
that the sign of $c_\tau$ is switched, and $\mu_\tau$ is replaced by
$\mu_\tau'$.  The usual considerations using equations
\eqref{eqn:cTriv}, \eqref{eqn:QTriv}, and \eqref{eqn:CZTriv} show that
$J_0$ does not depend on $\tau$.

\begin{example}
  Suppose that $C$ is an embedded holomorphic curve in $X$, whose ends
  are at distinct Reeb orbits with multiplicity $1$.  Then $w_\tau(C)
  =0$ and $\mu_\tau' (C)=0$, so it follows from the relative
  adjunction formula \eqref{eqn:RAF} that
\[
J_0(C) = -\chi(C).
\]
\end{example}

Two variants of $J_0$ are also of interest.

\begin{definition}
\label{def:JVariant}
If $\alpha=\{(\alpha_i,m_i)\}$ is an orbit set, define
\[
|\alpha| \eqdef \sum_i\left\{\begin{array}{cl} 1, & \mbox{$\alpha_i$
      elliptic}, \\
m_i, & \mbox{$\alpha_i$ positive hyperbolic,}\\
\ceil{m_i/2}, & \mbox{$\alpha_i$ negative hyperbolic.}
\end{array}\right.
\]
Now define
\[
\begin{split}
  J_+(\alpha^+,\alpha^-,Z) &\eqdef J_0(\alpha^+,\alpha^-,Z) +
  |\alpha^+| -
  |\alpha^-|,\\
  J_-(\alpha^+,\alpha^-,Z) &\eqdef J_0(\alpha^+,\alpha^-,Z) -
  |\alpha^+| + |\alpha^-|.
\end{split}
\]
\end{definition}

\subsection{Properties of $J_0$, $J_+$, and $J_-$ in symplectizations}
\label{sec:JDiscussion}

Suppose now that $X$ is the symplectization of a closed oriented
3-manifold $Y$ with a stable Hamiltonian structure.  Then $J_0$, $J_+$, and
$J_-$ have the following basic properties, which are similar to those
of the ECH index $I$:

\begin{proposition}
\label{prop:JBasics}
Suppose $X$ is the symplectization of $Y$.  Fix $J$ to denote one of
$J_0$, $J_+$, or $J_-$.  Then:
\begin{description}
\item{(a)}
(Additivity) If $Z\in H_2(Y,\alpha,\beta)$ and $W\in
  H_2(Y,\beta,\gamma)$ then
\[
J(\alpha,\beta,Z) + J(\beta,\gamma,W) = J(\alpha,\gamma,W).
\]
\item{(b)} (Ambiguity) If $Z,Z'\in H_2(Y,\alpha,\beta)$ where
$[\alpha]=[\beta]=\Gamma\in H_1(Y)$ then
\[
J(\alpha,\beta,Z) - J(\alpha,\beta,Z') = \langle -c_1(\xi) +
2\op{PD}(\Gamma), Z-Z'\rangle.
\]
\item{(c)} (Absolute version) For each orbit set
  $\alpha=\{(\alpha_i,m_i)\}$, there is a homotopy class of oriented
  2-plane fields $J(\alpha) \in \mc{P}(Y)$ such that:
\begin{description}
\item{(i)} $J(\alpha)$ is obtained by modifying $\xi$ by a canonical
  manner (up to homotopy, depending only on $m_i$) in disjoint tubular
  neighborhoods of each $\alpha_i$.
\item{(ii)} $\frak{s}(J(\alpha)) = \frak{s}(\xi) - \op{PD}([\alpha])$.
\item{(iii)} If $\alpha$ and $\beta$ are orbit sets with
  $[\alpha]=[\beta]=\Gamma$ then $J(\alpha,\beta,Z)\equiv
  J(\alpha) - J(\beta)$ in $\Z/d(-c_1(\xi)+2\op{PD}(\Gamma))$.
\end{description}
\end{description}
\end{proposition}

\begin{proof}
One mimics the proofs of the corresponding properties of $I$.  For
part (c), one replaces equation \eqref{eqn:DAG} by
\[
J_0(\alpha) \eqdef P'(L) -
\sum_iw_{\tau_i}(\zeta_i)+\mu_\tau'(\alpha),
\]
where $P'(L)$ was defined in Remark~\ref{rem:P'}.
\end{proof}

The culminating result to be proved in this section is:

\begin{theorem}
\label{thm:J+}
Suppose $X$ is the symplectization of a contact 3-manifold $Y$, with
an admissible almost complex structure.  Then every holomorphic curve
$C$ in $X$ satisfies $J_+(C)\ge 0$.
\end{theorem}

Note that the almost complex structure is not assumed to be generic,
and the holomorphic curve $C$ is not assumed to be simple or irreducible.

\begin{remark}
  This theorem implies that the differential $\partial$ in the
  embedded contact homology (or any other kind of contact homology) of
  a contact 3-manifold can be decomposed as $
  \partial = \partial_0 + \partial_1 + \cdots $ where $\partial_k$
  denotes the contribution from holomorphic curves $C$ with
  $J_+(C)=k$.  By the additivity of $J_+$, the identity $\partial^2=0$
  can be refined to $\partial_0^2=0$,
  $\partial_0\partial_1+\partial_1\partial_0$, etc.  However it is not
  clear if this leads to new topological invariants, except perhaps in
  some special situations, because in general maps induced by cobordisms
  might include contributions from curves with $J_+$ negative.
\end{remark}

\begin{example}
  A tool that was used in \cite{t3} to help compute the embedded
  contact homology of $T^3$ turns out to be a special case of the
  relative filtration $J_+$.  Namely, for the contact forms on $T^3$
  considered in \cite{t3}, $J_+(\alpha,\beta,Z)$ equals
  $I(\alpha,\beta,Z)$ plus the number of hyperbolic orbits in $\alpha$
  minus the number of hyperbolic orbits in $\beta$.  All curves that
  contribute to the ECH differential in this case have $J_+=2$.  Thus
  $2I-J_+$ defines a second grading which is preserved by the
  differential, and this is what appears in \cite[Def. 5.1]{t3}.
\end{example}

\subsection{Lower bounds on $J_0$ in the general case}

Theorem~\ref{thm:indI} has the following analogue for $J_0$.

\begin{proposition}
\label{prop:JBound}
  Let $X$ be a symplectic cobordism from $Y_+$ and $Y_-$, let
  $\alpha^+$ be an orbit set in $Y_+$, and let $\alpha^-$ be an orbit
  set in $Y_-$.  Suppose $C\in\mc{M}(\alpha^+,\alpha^-)$ is simple and
  irreducible and has genus $g$.  Then
\begin{equation}
\label{eqn:JBound}
J_0(C) \ge 2(g - 1 + \delta(C)) + \sum_{\gamma}
\left\{\begin{array}{cl} 2n_\gamma - 1, & \mbox{$\gamma$ elliptic,} \\
m_\gamma, & \mbox{$\gamma$ positive hyperbolic},\\
\frac{m_\gamma + n_\gamma^{odd}}{2}, & \mbox{$\gamma$ negative
  hyperbolic}.
\end{array}
\right.
\end{equation}
Here the sum is over all embedded Reeb orbits $\gamma$ in $Y_+$ or
$Y_-$ at which $C$ has ends; $m_\gamma$ denotes the total multiplicity
of the ends of $C$ at $\gamma$; $n_\gamma$ denotes the number of ends
of $C$ at $\gamma$; and $n_\gamma^{odd}$ denotes the number of ends of
$C$ at $\gamma$ with odd multiplicity\footnote{When $X$ is a
  symplectization, $Y_+$ and $Y_-$ are still regarded as distinct,
  i.e.\ the positive and negative ends are to be counted in separate
  summands in \eqref{eqn:JBound}.}.
\end{proposition}

\begin{proof}
  Let $n$ denote the number of ends of $C$.  The relative adjunction
  formula \eqref{eqn:RAF} implies that
\[
J_0(C) = 2g - 2 + n + \mu_\tau'(C) - w_\tau(C) + 2\delta(C).
\]
Thus it is enough to show that
$n + \mu_\tau'(C) - w_\tau(C)$ is greater than or equal to the sum
over $\gamma$ in \eqref{eqn:JBound}.

To prove this last inequality, we can assume without loss of
generality (as will be clear from the argument below) that there is a
single embedded Reeb orbit $\gamma$ in $Y_+$ such that all ends of $C$
are positive ends at $\gamma$.  Thus $\alpha^-=\emptyset$, and we can
write $\alpha^+=\{(\gamma,m)\}$.  Let $q_1,\ldots,q_n$ denote the
multiplicities of the positive ends of $C$ at $\gamma$, so in
particular $\sum_{i=1}^nq_i=m$.  Let $\zeta_1,\ldots,\zeta_n$ denote
the corresponding braids around $\gamma$, and let $\zeta\eqdef
\bigcup_i\zeta_i$.  Then we need to show that
\begin{equation}
\label{eqn:NTSI}
n + \sum_{k=1}^{m-1}\op{CZ}_\tau(\gamma^k) - w_\tau(\zeta)
\ge \left\{\begin{array}{cl} 2n-1, & \mbox{$\gamma$ elliptic},\\
m, & \mbox{$\gamma$ positive hyperbolic,}\\
\frac{m+n_{odd}}{2}, & \mbox{$\gamma$ negative hyperbolic.}
\end{array}
\right.
\end{equation}
Here $n_{odd}$ denotes the number of odd $q_i$'s.

Case 1: $\gamma$ is elliptic.  In this case the proof of
\eqref{eqn:NTSI} follows the proof of the inequality in
Lemma~\ref{lem:UWB}.  However instead of the combinatorial inequality
in Lemma~\ref{lem:CE1}, one needs the slightly stronger inequality
\[
\sum_{i,j=1}^n\max(q_i\floor{q_j\theta},q_j\floor{q_i\theta}) \le
2\sum_{k=1}^{m-1}\floor{k\theta} + \sum_{i=1}^n\floor{q_i\theta} + m - n.
\]
This is proved the same way as Lemma~\ref{lem:CE1}, but with the
inequality \eqref{eqn:pick3} replaced by the equally obvious inequality
\[
L \le 2 + \sum_{k=1}^{m-1}(\floor{k\theta}+1) +
\sum_{i=1}^n\floor{q_i\theta}.
\]

Case 2: $\gamma$ is positive hyperbolic.  In this case
Lemma~\ref{lem:WBE} can be improved to 
\[
w_\tau(\zeta_i) \le (\rho_i-1)(q_i -1).
\]
This can be proved by the arguments in \cite[Lem. 6.8]{pfh2}, using
\cite[Thm. 2.3]{siefring1} to provide the necessary asymptotic
analysis in the present setting.
Hence the writhe bound in Lemma~\ref{lem:UWB} can be improved to
\[
w_\tau(\zeta) \le \sum_{k=1}^m\op{CZ}_\tau(\gamma^k) -
\sum_{i=1}^n\op{CZ}_\tau(\gamma^{q_i}) - \sum_{i=1}^n(q_i-1).
\]
So to prove \eqref{eqn:NTSI} in this case it is enough to show that
\[
n - \op{CZ}_\tau(\gamma^m) + \sum_{i=1}^n\op{CZ}_\tau(\gamma^{q_i}) +
\sum_{i=1}^n (q_i-1)
\ge m.
\]
But this is an equality, because by \eqref{eqn:CZHyp}, all the
CZ terms cancel.

Case 3: $\gamma$ is negative hyperbolic.  Choose the trivialization
$\tau$ so that $\op{CZ}(\tau^k)=k$.  In this case Lemma~\ref{lem:WBE} can
be improved to
\[
w_\tau(\zeta_i) \le \ceil{\frac{(q_i-1)^2}{2}},
\]
again by the arguments in \cite[Lem. 6.8]{pfh2} with the help of
\cite[Thm. 2.3]{siefring1}.  So together with
Lemma~\ref{lem:linkingBound}, we obtain
\[
w_\tau(\zeta) \le \sum_{i=1}^n \ceil{\frac{(q_i-1)^2}{2}} +
2\sum_{i=1}^n \sum_{j=i+1}^n \max\left(q_i\floor{\frac{q_j}{2}},
  q_j\floor{\frac{q_i}{2}}\right).
\]
To simplify this, order the $q_i$'s so that $q_1,\ldots,q_{n_{odd}}$
are odd and $q_1\ge\cdots\ge q_{n_{odd}}$.  A straightforward
calculation then deduces from the above inequality that
\[
n + \sum_{k=1}^{m-1} \op{CZ}_\tau(\gamma^k) - w_\tau(\zeta) \ge
\frac{m+n_{odd}}{2} + \sum_{j=1}^{n_{odd}}(j-1)q_j.
\]
The inequality \eqref{eqn:NTSI} follows.
\end{proof}

\begin{corollary}
\label{cor:EulerBound}
If $C$ is a simple holomorphic curve as above, then
\[
-\chi(C) \le J_0(C) - 2\delta(C).
\]
\end{corollary}

\begin{proof}
  We just have to check that for each $\gamma$ at which $C$ has ends,
  the corresponding summand in \eqref{eqn:JBound} is at least
  $n_\gamma$.  But this is easy.  (For the negative hyperbolic case,
  note that since each end has multiplicity at least one if odd and at
  least two if even, we have $m_\gamma \ge
  n_\gamma^{odd}+2(n_\gamma-n_\gamma^{odd})$.)  The argument works
  just as well if $C$ is not irreducible, as long as it is simple.
\end{proof}

By a similar but even easier argument, we have:

\begin{corollary}
\label{cor:J+}
  If $C\in\mc{M}(\alpha^+,\alpha^-)$ is a simple irreducible
  holomorphic curve as above, then
\[
J_0(C) \ge 2(g-1+\delta(C)) + |\alpha^+| + |\alpha^-|,
\]
or equivalently
\[
J_\pm(C) \ge 2(g-1+|\alpha^\pm|+\delta(C)).
\]
\end{corollary}

\begin{remark}
This last inequality shows that $J_+$ is similar to the relative
filtration on the symplectic field theory \cite{egh} of a contact
manifold given by genus plus number of positive ends minus one.
\end{remark}

\begin{remark}
The inequality \eqref{eqn:JBound} implies the index inequality
\eqref{eqn:indI}.  One can see this by adding the index formula
\eqref{eqn:ind} and then arguing as in the proofs of the above
corollaries.  In particular, if $\op{ind}(C)=I(C)$ (e.g.\ if $C$ is a
curve in the symplectization of a contact manifold that contributes to
the ECH differential and does not contain trivial cylinders), then the
inequality \eqref{eqn:JBound} is sharp.
\end{remark}

\subsection{$J_0$ of unions and multiple covers}

As usual, let $X$ be a symplectic cobordism from $Y_+$ to $Y_-$ with
an admissible almost complex structure.

\begin{proposition}
\label{prop:JU}
If $C$ and $C'$ are holomorphic curves in $X$, then
\[
J_0(C\cup C') \ge J_0(C) + J_0(C') + 2C\cdot C' + E + N,
\]
where\footnote{When $X$ is a symplectization, $Y_+$ and $Y_-$ are
  still regarded as distinct in the definition of $E$ and $N$.}:
\begin{itemize}
\item
$E$ denotes the number of elliptic Reeb orbits in $Y_+$ or $Y_-$ at which both
$C$ and $C'$ have ends.
\item $N$ denotes the number of negative hyperbolic orbits $\gamma$ in
  $Y_+$ or $Y_-$ such that the total multiplicity of the ends of $C$
  at $\gamma$ and the total multiplicity of the ends of $C'$ at
  $\gamma$ are both odd.
\end{itemize}
\end{proposition}

\begin{proof}
  One copies the proof of Theorem~\ref{thm:IU} with minor
  modifications.  In particular, the same calculation as before shows
  that equation \eqref{eqn:PTAT} holds with $I$ replaced by $J_0$ and with
  $\mu_\tau$ replaced by $\mu_\tau'$.  To prove that the right hand
  side of this modified equation \eqref{eqn:PTAT} is at least $E+N$,
  one follows the proof of \eqref{eqn:huge}, but replacing
  Lemma~\ref{lem:CLI} with Lemma~\ref{lem:CLI'} below.
\end{proof}

\begin{lemma}
\label{lem:CLI'}
Under the assumptions of Lemma~\ref{lem:CLI}, we have
\begin{equation}
\label{eqn:CLI'}
2\sum_{i=1}^n\sum_{j=1}^{n'}\max(q_i\rho'_j,\rho_iq'_j) \le
\left(\sum_{k=1}^{m+m'-1} - \sum_{k=1}^{m-1} -
  \sum_{k=1}^{m'-1}\right)\op{CZ}_\tau(\gamma^k).
\end{equation}
If $\gamma$ is elliptic and $m,m'>0$, or if $\gamma$ is negative hyperbolic and
both $m$ and $m'$ are odd, then the inequality is strict.
\end{lemma}

\begin{proof}
  We slightly modify the proof of Lemma~\ref{lem:CLI} as follows.  If
  $\gamma$ is hyperbolic, then the inequality \eqref{eqn:CLI'} is
  equivalent to \eqref{eqn:clii} because the right hand sides are
  equal by \eqref{eqn:CZHyp}.  If $\gamma$ is negative hyperbolic and
  both $m$ and $m'$ are odd, then there is a pair $(i,j)$ such that
  $q_i$ and $q_j'$ are both odd so that the inequality
  \eqref{eqn:obviousInequality} is strict, so \eqref{eqn:clii} is
  strict.

  Now suppose $\gamma$ is elliptic with monodromy angle $\theta$.
  Without loss of generality $m,m'>0$.  By equation \eqref{eqn:CZEll},
  the right hand side of \eqref{eqn:CLI'} minus the right hand side of
  \eqref{eqn:clii} equals
\[
1 - 2\left(\floor{(m+m')\theta} - \floor{m\theta} -
  \floor{m'\theta}\right) \in \{-1,1\}.
\]
If this is $1$ then we are done.  If this is $-1$, then equality does
not hold in \eqref{eqn:inductiveStep}, so the two sides of
\eqref{eqn:clii} differ by at least $2$ and we are also done.
\end{proof}

\subsection{The relative filtration $J_+$}

\begin{proof}[Proof of Theorem~\ref{thm:J+}.]
  Since $J_+(C)$ depends only on the relative homology class of $C$,
  we may assume that $C$ is a union of $k$ (not necessarily distinct)
  simple, irreducible holomorphic curves.  We now prove the theorem by
  induction on $k$.

If $C$ is simple and irreducible, then $J_+(C)\ge 0$ by
Corollary~\ref{cor:J+}, since the assumption that $Y$ is a
contact manifold guarantees that $C$ has at least one positive end.
Also, if $C$ is any multiple cover of a cylinder $\R\times\gamma$
where $\gamma$ is a Reeb orbit, then $J_+(C)=0$ by definition.

To complete the induction, it is enough to show that if
$C\in\mc{M}(\alpha,\beta)$ and $C'\in\mc{M}(\alpha',\beta')$ satisfy
$J_+(C), J_+(C')\ge 0$, and if the images of $C$ and $C'$ do not have
a cylinder $\R\times\gamma$ in common, then $J_+(C\cup C')\ge 0$.
Note that $C\cdot C'\ge 0$ by intersection positivity and
Proposition~\ref{prop:CdotC}.  So by Proposition~\ref{prop:JU},
\[
J_+(C\cup C') \ge E + N + (|\alpha\alpha'|-|\alpha|-|\alpha'|) -
(|\beta\beta'|-|\beta|-|\beta'|).
\]
Now write
\[
E = E_+ + E_-, \quad\quad N = N_+ + N_-,
\]
where $E_+$ denotes the number of elliptic orbits that appear in both
$\alpha$ and $\alpha'$; $E_-$ denotes the number of elliptic orbits
that appear in both $\beta$ and $\beta'$; $N_+$ denotes the number of
negative hyperbolic orbits that appear with odd multiplicity in both
$\alpha$ and $\alpha'$; and $N_-$ denotes the number of negative
hyperbolic orbits that appear with odd multiplicity in both $\beta$
and $\beta'$.  It follows from Definition~\ref{def:JVariant} that
\[
\begin{split}
|\alpha\alpha'| &= |\alpha| + |\alpha'| - E_+ - N_+,\\
|\beta\beta'| &= |\beta| + |\beta'| - E_- - N_-.
\end{split}
\]
Putting the above together gives $J_+(C\cup C') \ge 2(E_-+N_-) \ge 0$.
\end{proof}

%

\begin{acknowledgments}
  The essential importance of the work of Cliff Taubes for this
  project should be self-evident.  In addition I thank Robert
  Lipshitz, Richard Siefring, and Chris Wendl for some helpful
  conversations, and Yasha Eliashberg for his continued guidance,
  inspiration, and support throughout the development of ECH.
\end{acknowledgments}

\end{document}